\documentclass[12pt,a4paper,leqno,noamsfonts]{amsart}
\usepackage[english]{babel}
\usepackage[dvipsnames]{xcolor}
\usepackage{graphicx,pifont,soul}
\usepackage[utopia]{mathdesign} 
\usepackage[a4paper,top=3cm,bottom=3.3cm,
           left=3cm,right=3.3cm,marginparwidth=60pt]{geometry}
\usepackage[utf8]{inputenc}
\usepackage{braket,caption,comment,mathtools,stmaryrd,physics}
\usepackage[usestackEOL]{stackengine}
\usepackage{multirow,booktabs,microtype,relsize}
\usepackage[colorlinks,bookmarks]{hyperref} %
      \hypersetup{colorlinks,%
            citecolor=britishracinggreen,%
            filecolor=black,%
            linkcolor=cobalt,%
            urlcolor=cornellred}
      \numberwithin{equation}{section}
\usepackage[capitalise]{cleveref}

\DeclareSymbolFont{usualmathcal}{OMS}{cmsy}{m}{n}
\DeclareSymbolFontAlphabet{\mathcal}{usualmathcal}
\makeatletter
\newcommand{\mylabel}[2]{#2\def\@currentlabel{#2}\label{#1}}
\makeatother
\usepackage[colorinlistoftodos]{todonotes}

\definecolor{cornellred}{rgb}{0.7, 0.11, 0.11}
\definecolor{britishracinggreen}{rgb}{0.0, 0.26, 0.15}
\definecolor{cobalt}{rgb}{0.0, 0.28, 0.67}

\newcommand{\BA}{{\mathbb{A}}}

\newcommand{\BG}{{\mathbb{G}}}

\newcommand{\BN}{{\mathbb{N}}}

\newcommand{\BP}{{\mathbb{P}}}
\newcommand{\BQ}{{\mathbb{Q}}}

\newcommand{\BZ}{{\mathbb{Z}}}
\newcommand{\CA}{{\mathcal A}}

\newcommand{\CE}{{\mathcal E}}
\newcommand{\CF}{{\mathcal F}}
\newcommand{\CG}{{\mathcal G}}

\newcommand{\CL}{{\mathcal L}}

\newcommand{\CS}{{\mathcal S}}

\newcommand{\CX}{{\mathcal X}}

\newcommand{\FS}{{\mathfrak S}}
\newcommand{\bfk}{\mathbf{k}}

\newcommand{\simto}{\,\widetilde{\to}\,}
\newcommand{\into}{\hookrightarrow}
\newcommand{\onto}{\twoheadrightarrow}

\newcommand{\CCoh}{\mathscr{C}\kern-0.25em {o}\kern-0.2em{h}}
\newcommand{\boldit}[1]{\boldsymbol{#1}}

\newcommand{\Isom}{\mathbf{Isom}}
\newcommand*{\defeq}{\mathrel{\vcenter{\baselineskip0.5ex \lineskiplimit0pt
                     \hbox{\scriptsize.}\hbox{\scriptsize.}}}%
                     =}

\DeclareMathOperator{\cycle}{cycle}

\DeclareMathOperator{\Quot}{Quot}
\DeclareMathOperator{\red}{red}
\DeclareMathOperator{\nilp}{nilp}

\DeclareMathOperator{\Sym}{Sym}
\DeclareMathOperator{\Coh}{Coh}

\DeclareMathOperator{\QCoh}{QCoh}

\DeclareMathOperator{\GL}{GL}

\DeclareMathOperator{\length}{length}

\DeclareMathOperator{\pr}{pr}
\DeclareMathOperator{\Mor}{Mor}

\DeclareMathOperator{\Rep}{Rep}

\newcommand{\cE}{\curly E}
\newcommand{\cF}{\curly F}
\newcommand{\cG}{\curly G}

\DeclareFontFamily{OT1}{rsfs}{}
\DeclareFontShape{OT1}{rsfs}{n}{it}{<-> rsfs10}{}
\DeclareMathAlphabet{\curly}{OT1}{rsfs}{n}{it}

\newcommand\Hom{\operatorname{Hom}}
\newcommand{\lHom}{\mathscr{H}\kern-0.3em {o}\kern-0.2em{m}}
\newcommand{\RHom}{{\mathbf{R}\kern-0.07em\mathrm{Hom}}}
\newcommand{\RRlHom}{\mathbf{R}\kern-0.025em\mathscr{H}\kern-0.3em {o}\kern-0.2em{m}}
\newcommand{\Ext}{\operatorname{Ext}}
\newcommand{\lExt}{{\mathscr{E}\kern-0.2em xt}}
\newcommand{\End}{\operatorname{End}}
\newcommand{\lEnd}{\mathscr{E}\kern-0.15em {n}\kern-0.1em{d}}

\newcommand{\id}{\operatorname{id}}
\newcommand{\Spec}{\operatorname{Spec}}
\newcommand{\Proj}{\operatorname{Proj}}
\newcommand{\Supp}{\operatorname{Supp}}
\newcommand{\supp}{\mathsf{supp}}

\newcommand{\Pic}{\mathop{\rm Pic}\nolimits}

\newcommand{\HH}{\mathrm{H}}
\newcommand{\OO}{\mathscr O}
\newcommand{\inj}{\mathsf{inj}}



\usepackage[all]{xy}
\usepackage{tikz}
\usepackage{tikz-cd}
\usepackage{adjustbox}
\usepackage{rotating}
\usepackage{comment}

\usetikzlibrary{matrix,shapes,intersections,arrows,decorations.pathmorphing}
\tikzset{commutative diagrams/arrow style=math font}
\tikzset{commutative diagrams/.cd,
mysymbol/.style={start anchor=center,end anchor=center,draw=none}}
\newcommand\MySymb[2][\square]{%
  \arrow[mysymbol]{#2}[description]{#1}}
\tikzset{
shift up/.style={
to path={([yshift=#1]\tikztostart.east) -- ([yshift=#1]\tikztotarget.west) \tikztonodes}
}
}

\theoremstyle{definition}

\newtheorem*{lemma*}{Lemma}
\newtheorem*{theorem*}{Theorem}
\newtheorem*{example*}{Example}
\newtheorem*{fact*}{Fact}
\newtheorem*{notation*}{Notation}
\newtheorem*{definition*}{Definition}
\newtheorem*{prop*}{Proposition}
\newtheorem*{remark*}{Remark}
\newtheorem*{corollary*}{Corollary}
\newtheorem*{conventions*}{Conventions}

\newtheorem{definition}{Definition}[section]

\newtheorem{example}[definition]{Example}

\newtheorem{notation}[definition]{Notation}
\newtheorem{exercise}[definition]{Exercise}
\newtheorem{remark}[definition]{Remark}

\newtheoremstyle{thm} 
        {3mm}
        {3mm}
        {\slshape}
        {0mm}
        {\bfseries}
        {.}
        {1mm}
        {}
\theoremstyle{thm}
\newtheorem{theorem}[definition]{Theorem}
\newtheorem{corollary}[definition]{Corollary}
\newtheorem{lemma}[definition]{Lemma}
\newtheorem{prop}[definition]{Proposition}
\newtheorem{thm}{Theorem}

\newcommand{\Ab}{\mathbb A}

\newcommand{\cA}{\mathcal A}

\newcommand{\cL}{\mathcal L}

\title[On the stack of 0-dimensional coherent sheaves: structural aspects]{On the stack of 0-dimensional coherent sheaves: \\ structural aspects}
\author{Barbara Fantechi, Andrea T. Ricolfi}

\begin{document}

\begin{abstract}
Let $X$ be a quasiprojective scheme. In this expository note we collect a series of useful structural results on the stack $\CCoh^n(X)$ parametrising $0$-dimensional coherent sheaves of length $n$ over $X$. For instance, we discuss its functoriality (in particular its behaviour along \'etale maps), the support morphism to $\Sym^n(X)$, and its relationship with the Quot scheme of points $\Quot_X(\CE,n)$ for fixed $\CE\in \Coh(X)$.
\end{abstract}

\maketitle

{\hypersetup{linkcolor=black}\tableofcontents}

\section{Introduction}

Let $X$ be a quasiprojective scheme defined over an algebraically closed field $\bfk$, and let $n\ge 0$ be an integer. The purpose of this expository paper is to provide a self-contained introduction to the algebraic stacks 
\[
\CCoh^n(X)
\]
parametrising coherent sheaves $\CF \in \Coh(X)$ with 0-dimensional support and length $\chi(\CF)=n$. 

Our main aim is to provide a readable introduction to this topic, with the purpose of helping researchers new to the field in approaching the current literature.

Our viewpoint, motivating our interest in these stacks, is that whenever a moduli space (such as Hilbert schemes and Quot schemes of points) involves a flat family of $0$-dimensional, coherent sheaves with proper support, by definition it has a natural morphisms into $\CCoh^n(X)$ --- see e.g.~\Cref{lemma:quot-to-coh} --- and to successfully study such moduli spaces (e.g.~via their invariants) one might want to deal with $\CCoh^n(X)$ first and then pull the results back to the moduli space of interest. This is for instance the point of view of \cite{Huang-Jiang}, which provides an interesting motivic relation (see Theorem 1.7 in loc.~cit) between punctual Quot schemes on a variety and the punctual stacks of coherent sheaves over it.

\smallbreak
We now briefly summarise the main contents of this paper.

Let $X$ be a quasiprojective scheme over $\bfk$. In \Cref{sec:families} we introduce families of coherent sheaves over $X$ with proper support, leading directly to the definition of the stack $\CCoh(X)$ of coherent sheaves in \Cref{sec:coh(X)}. We then proceed defining three important \emph{open substacks} of $\CCoh(X)$, namely
\begin{itemize}
\item [\mylabel{substack-1}{(1)}] the open substack $\CCoh(X)_U \subset \CCoh(X)$ of sheaves with support contained in a fixed open subscheme $U\subset X$ (cf.~\Cref{lemma:supp-U}),
\item [\mylabel{substack-2}{(2)}] the open substack $\CCoh_m(X) \subset \CCoh(X)$ of $m$-regular sheaves in the sense of Castelnuovo--Mumford (cf.~\Cref{prop:m-reg-is-open}), and 
\item [\mylabel{substack-3}{(3)}] the open (and closed) substack $\CCoh^P(X) \subset \CCoh(X)$ of sheaves with fixed Hilbert polynomial $P \in \BQ[z]$ (cf.~\Cref{prop:hilb-poly-constant}).
\end{itemize}
To make sense of \ref{substack-2} and \ref{substack-3}, we require $X$ to be projective.
We then prove the following structural result, realising $\CCoh^P(X)$ as a direct limit of its open substacks $\CCoh_m^P(X) = \CCoh^P(X) \cap \CCoh_m(X)$.

\begin{thm}[\Cref{thm:global-quot-1}]\label{mainthm1}
Let $X$ be a projective variety, $P \in \BQ[z]$ a polynomial. The stack $\CCoh^P(X)$ has an exhaustion by open substacks
\[
\cdots \into \CCoh^P_m(X) \into \CCoh^P_{m+1}(X) \into \cdots
\] 
where each open substack $\CCoh^P_m(X)$ is a global quotient
\[
\CCoh^P_m(X) \cong [U_m^P / \GL_{P(m)}]
\]
of the open subscheme $U_m^P \subset \Quot_{X}(\OO_{X}(-m)^{\oplus P(m)},P)$ parametrising isomorphism classes of surjections $\theta \colon \OO_{X}(-m)^{\oplus P(m)} \onto \CF$ such that $\CF$ is $m$-regular and $\HH^0(X,\theta(m))$ is an isomorphism.
\end{thm}

In \Cref{sec:0-dim-sheaves} we define, for a quasiprojective scheme $X$, the stack $\CCoh^n(X)$ of $0$-dimensional coherent sheaves of length $n$. We realise it as a natural open substack of $\CCoh(X)$. 
Then \Cref{mainthm1} directly yields the following global quotient description of $\CCoh^n(X)$. 

\begin{thm}[\Cref{prop:Coh^n-global-quotient}]\label{mainthm2}
Let $X$ be a projective variety, $n\geq 0$ an integer. There is an open $\GL_{n}$-invariant subscheme $A_{X,n}\subset \Quot_X(\OO_X^{\oplus n},n)$ along with an isomorphism of stacks
\[
\begin{tikzcd}
    {[}A_{X,n}/\GL_{n}{]} \arrow{r}{\sim} & \CCoh^n(X).
\end{tikzcd}
\]
In particular, $\CCoh^n(X)$ is algebraic, with $A_{X,n}$ serving as a smooth atlas.
\end{thm}
In \Cref{sec:Coh(A^d)} we focus on the case $X=\BA^d$, and we establish an isomorphism of stacks (cf.~\Cref{thm:affine-coh^n})
\begin{equation}\label{eqn:iso-A^d}
\begin{tikzcd}
\CCoh^n(\Ab^d) \arrow{r}{\sim} & {[}C(n,d)/\GL_{n}{]},
\end{tikzcd}
\end{equation}
where $C(n,d)\subset \BA^{dn^2}$ is the \emph{commuting variety}, i.e.~the scheme parametrising $d$-tuples of pairwise commuting $n\times n$ matrices.

In \Cref{sec:naive-coh-to-chow} we construct the `naive Coh-to-Chow morphism' $\CCoh^n(X)_{\red} \to \Sym^n(X)$ to the symmetric product $\Sym^n(X) = X^n / \FS_n$. It sends a sheaf $[\CF] \in\CCoh^n(X)$ to the associated $0$-cycle
\[
\cycle(\CF) = \sum_{x \in X}\length_{\OO_{X,x}}\CF_x\cdot x \in \Sym^n(X).
\]
We refer to Rydh \cite[IV, Prop.~7.8]{Rydh1} for the construction of the `actual' Coh-to-Chow morphism 
\begin{equation}\label{eqn:support-map}
\begin{tikzcd}
\supp_X^n\colon \CCoh^n(X) \arrow{r} & \Sym^n(X).
\end{tikzcd}    
\end{equation}

Note that, for any open subscheme $U \subset X$, the stack $\CCoh^n(U)$ agrees with the open substack $\CCoh^n(X)_U = \CCoh^n(X) \times_{\Sym^n(X)}\Sym^n(U)$ introduced in \ref{substack-1} above.

In \Cref{sec:etale-behaviour} we study the functorial behaviour of $\CCoh^n(-)$ along \'etale morphisms. A key technical lemma we establish (cf.~\Cref{lemma:etale1}) says that if $f\colon Y \to X$ is an \'etale morphism of varieties, the open locus $V_{\inj} \subset \CCoh^n(Y)$ of sheaves $\CE$ such that $f$ is injective on the set-theoretic support of $\CE$ admits an \emph{\'etale} morphism $f_\ast \colon V_{\inj} \to \CCoh^n(X)$, given on points by the pushforward. This is the key to a series of useful applications, some of which will be treated in our companion paper \cite{Coh(X)-motivic}. One application we discuss here is the following: if $X$ is smooth of dimension $d$, then $\CCoh^n(X)$ is smooth if and only if $\CCoh^n(\BA^d)$ is smooth. This, in turn, combined with the isomorphism \eqref{eqn:iso-A^d}, immediately yields the following characterisation of smoothness for $\CCoh^n(X)$.

\begin{thm}[\Cref{thm:smoothness-of-Coh}]
Let $X$ be a smooth $d$-dimensional variety. Then $\CCoh^n(X)$ is smooth if and only if either $d=1$ or $n=1$.
\end{thm}

In \Cref{sec:stratification} we revisit the classical stratification of $\CCoh^n(X)$ in terms of partitions $\alpha\vdash n$, for any smooth variety $X$. This stratification is induced by the one on $\Sym^n(X)$ via pullback along the Coh-to-Chow morphism \eqref{eqn:support-map}. Note that the naive version $\CCoh^n(X)_{\red} \to \Sym^n(X)$ of the support map \eqref{eqn:support-map} is enough to study this stratification, at least from a motivic perspective. In \Cref{prop:local-triviality} we establish the \'etale-local triviality of the `$\alpha$-component' of the support map 
\begin{equation}\label{alpha-component}
\begin{tikzcd}
\CCoh^n_\alpha(X) \arrow{r} & \Sym^n_\alpha(X).    
\end{tikzcd}
\end{equation}
The fibre is naturally identified with the product
\[
\prod_i \CCoh^i(\BA^d)_0^{\alpha_i}
\]
of \emph{punctual stacks} (parametrising sheaves entirely supported on the origin $0 \in \BA^d$), each of which is a global quotient stack (cf.~\Cref{sec:coh-to-chow-affine-and-punctual}) via
\[
\CCoh^n(\BA^d)_0 \cong [C(n,d)^{\nilp}/\GL_{n}].
\]
We shall see in our companion paper \cite[App.~A]{Coh(X)-motivic} an example showing that the \'etale locally trivial fibration \eqref{alpha-component} is \emph{not}, in general, \emph{Zariski} locally trivial. 

\subsection*{Acknowledgements} The authors are partially supported by the PRIN \emph{Geometry of algebraic structures: moduli, invariants, deformations} 2022BTA242. The first author is a member of GNSAGA of INDAM. 

This paper has been inspired by a series of lectures held by the first author at the BIRS-CMO Workshop held at Oaxaca, Mexico in 2022. She would like to thank the organisers for inviting her, and the participants for their lively interest.

Finally, we wish to thank an anonymous referee for suggesting several improvements on the first version of this paper.

\section{Conventions and background material}

\subsection{Conventions}
All schemes in this paper will be of finite type over a fixed algebraically closed field $\bfk$, and all morphisms $\bfk$-morphisms. We simply write `$X \times T$' for fibre products over $\Spec \bfk$. By point of a scheme we always mean $\bfk$-valued (i.e., closed) point; such points are dense in $X$. All group schemes, such as $\BG_m$ and more generally $\GL_n$, are defined over $\bfk$, although not explicitly displayed in the notation. 

Given a sheaf $\cF \in \Coh(X\times T)$ and a point $t \in T$, we denote by $\cF_t$ the coherent sheaf $\cF|_{X\times \set{t}}$ and we implicitly identify $X\times \set{t}$ with $X$ via the projection isomorphism, so that $\cF_t \in \Coh(X)$. 

Given a morphism of schemes $Y \to T$, and a coherent sheaf $\cF$ on $Y$, we say that $\cF$ is \emph{flat over} $T$, or $T$-\emph{flat}, if for any point $y \in Y$, with image $t \in T$, the stalk $\cF_y$ is flat as a module over $\OO_{T,t}$, via the canonical ring homomorphism $\OO_{T,t}\to \OO_{Y,y}$.

A morphism of schemes $f\colon Y \to T$ is \emph{quasiprojective} (resp.~projective) if there is an $f$-ample line bundle $\OO_Y(1)$ on $Y$ \cite[\href{https://stacks.math.columbia.edu/tag/01VV}{Tag 01VV}]{stacks-project} and \emph{projective} if it is quasiprojective and proper; equivalently, $f$ is projective if it can be obtained as $\Proj_{\OO_T}(\CA)$ as in \cite[II.7]{Hartshorne_AG}. A \emph{variety} will be an integral scheme quasiprojective over $\bfk$. For a scheme $Y$, we denote by $\Coh(Y)$, resp.~$\QCoh(Y)$, the abelian category of coherent, resp.~quasicoherent sheaves on $Y$. 

The \emph{Hilbert polynomial} $P_\CF$ of a coherent sheaf $\CF$ on a projective scheme $(X,\OO_X(1))$ is defined by $P_\CF(k)=\chi(\CF(k))$, where $\CF(k) = \CF \otimes_{\OO_X}\OO_X(1)^{\otimes k}$.

\subsection{Support of a coherent sheaf}
For any topological space $X$ and any sheaf $\CF$ of abelian groups on $X$, the {\em support} of $\CF$ is the closure in $X$ of the locus of points $p$ such that $\CF_p\neq 0$.

Note that in general the locus of points $p$ such that $\CF_p\neq 0$ is not closed, not even when $X$ is a scheme and $\CF$ is a quasicoherent sheaf. For instance, choose $\bfk=\mathbb{C}$, $Z\subset \BP^1$ an infinite, countable subset and let $\CF\in \QCoh(\BP^1)$ be the direct sum over $p\in Z$ of the skyscraper sheaves $\OO_p$.

However, if $\CF\in \Coh(X)$ is coherent and $\CF_p=0$ for some $p\in X$, then by Nakayama's lemma there is an open neighbourhood $U$ of $p$ in $X$ such that $\CF_q=0$ for every $q\in U$, so the set $\set{p\in X\,|\, \CF_p\neq 0}$ is closed in $X$. We call it the {\em set-theoretic support} of $\CF$ and we denote it by $\Supp_{\mathrm{set}}(\CF)$.

\begin{notation}
Let $X$ be a scheme and $\CF\in \Coh(X)$. We denote by $\Supp(\CF)$ the \emph{scheme-theoretic support} of $\CF$, namely the closed subscheme $\Supp(\CF) \into X$ whose ideal sheaf $\curly I\subset \OO_X$ is defined, on each open subset $U\subset X$, by 
    \[
    \curly I(U)\defeq\Set{f\in \OO_X(U)\,|\,f\cdot (-)\colon \CF|_U \to \CF|_U\mbox{ is identically 0}}.
    \]
The integer $\dim \Supp(\CF)$ is called the \emph{dimension} of $\CF$, and will be denoted $\dim \CF$. 
\end{notation}

\begin{exercise}
Let $X$ be a scheme, $\CF \in \Coh(X)$. Show that the underlying set of $\Supp(\CF)$ agrees with $\Supp_{\mathrm{set}}(\CF)$, so that $\Supp_{\mathrm{set}}(\CF)$ is closed in $X$. Show that if $f\colon X\to Y$ is any morphism of schemes and $\CF\in \Coh(Y)$, then $\Supp(f^*\CF)=f^{-1}\Supp(\CF)$ (if in need of a hint, you can consult \cite[\href{https://stacks.math.columbia.edu/tag/056H}{Tag 056H}]{stacks-project}). 
If $f$ is proper and $\CG\in \Coh(X)$, then $\Supp(f_*\CG)\subset f(\Supp \CG)$. Give an example to show that this inclusion is not, in general, an equality (hint: think of $X=\BP^1$ and $Y=\Spec \bfk$; can you find $\curly L\in \Pic(X)$ such that $f_*\curly L=0$?).
\end{exercise}

\begin{remark}
By definition, one can characterise $\Supp(\CF)$ as the smallest closed subscheme $i\colon Z\into X$ such that the natural homomorphism $\CF\to i_{*}(\CF|_Z) = i_\ast i^\ast \CF$ is an isomorphism.
\end{remark}

\subsection{Quot schemes}\label{sec:quot}
Let $(X,\OO_X(1))$ be a polarised projective scheme, $\CF \in \Coh(X)$ a coherent sheaf, $P \in \BQ[z]$ a polynomial (possibly constant). Grothendieck's Quot scheme 
\[
\Quot_X(\CF,P)
\]
is the (projective) scheme representing the functor sending a scheme $T$ to the set of isomorphism classes $[\CF_T \onto \cE]$ of quotients $\CF_T \onto \cE$ in $\Coh(X \times T)$, where $\cE$ is $T$-flat and satisfies $P_{\cE_t}=P$ for every $t \in T$. Here $\CF_T$ is the pullback of $\CF$ along $\pr_X\colon X \times T \to X$ and $P_{\cE_t}$ denotes the Hilbert polynomial of $\cE_{t}$ (with respect to $\OO_X(1)$). Finally, two quotients are isomorphic whenever they share the same kernel. We refer the reader to \cite{Grothendieck_Quot,fga_explained} for foundational material (such as existence and projectivity) regarding Quot schemes.

\section{The stack of coherent sheaves and some open substacks}\label{sec:stack-Coh(X)}

In this section we introduce the stack $\CCoh(X)$ of coherent sheaves on a scheme $X$, and some of its most natural open substacks. We refer the reader to \cite{MR2223406} for a thorough introduction to stacks in full generality, and to \cite{MR1905329} for a shorter survey. A detailed account on the stack of coherent sheaves, introduced and studied under minimal assumptions, can be found in the Stacks Project \cite[\href{https://stacks.math.columbia.edu/tag/08KA}{Tag 08KA}]{stacks-project}. See also \Cref{rmk:generalisations}.

\subsection{Families of coherent sheaves with proper support}\label{sec:families}
We start by defining \emph{families of coherent sheaves}, which will be the `points' of the stack $\CCoh(X)$ of coherent sheaves on a scheme $X$.

\begin{definition}\label{def:family-sheaves}
A \emph{family of coherent sheaves} on a scheme $X$, parametrised by a scheme $T$, is a sheaf $\curly F \in \Coh(X\times T)$, flat over $T$, such that $\Supp(\curly F)\to T$ is a proper morphism. An {\em isomorphism} of families  is an isomorphism $\alpha\colon \curly F\simto\curly G$ in $\Coh(X\times T)$. 
\end{definition}

\begin{remark}
 If $X$ is proper, in particular if $X$ is projective, then $\Supp(\cF)\to T$ is automatically a proper morphism. If $T=\Spec \bfk$ and $\dim \cF=0$, then $\Supp(\cF)$ is proper. Note that $\Supp(\cF) \to T$ need not be flat, even though $\cF$ is.
\end{remark}

\begin{exercise}
    Show that if $X$ is not proper, you can have flat families of 0-dimensional sheaves whose support is not proper. Hint: let $X=\mathbb A^1$, $T=\mathbb P^1$, and let $\cF=j_\ast \OO_\Delta$, the pushforward of the structure sheaf of the diagonal along $j\colon \BA^1 \times \BA^1 \into \BA^1\times\BP^1$.
\end{exercise}

We have the following useful criterion for flatness along a projective family (see also \Cref{prop:hilb-poly-constant}).

\begin{prop}\label{flatness-criteria}
Let $f\colon Y \to T$ be a projective morphism of schemes, with $T$ noetherian. Let $\CL$ be an $f$-ample line bundle on $Y$. Fix $\cF \in \Coh(Y)$. Consider the following conditions:
\begin{itemize}
    \item [\mylabel{flat-1}{(1)}] $\cF$ is $T$-flat.
    \item [\mylabel{flat-2}{(2)}] $f_{\ast} (\cF \otimes_{\OO_Y} \CL^{\otimes m})$ is locally free of finite rank for $m \gg 0$.
    \item [\mylabel{flat-3}{(3)}] The Hilbert polynomial
    \[
k\mapsto \chi(\cF_t \otimes_{\OO_{Y_t}}\CL_t^{\otimes k})
    \]
    is locally constant as a function on $T$.
\end{itemize}
Then {\normalfont{\ref{flat-1}}} is equivalent to {\normalfont{\ref{flat-2}}} and implies {\normalfont{\ref{flat-3}}}. All three conditions are equivalent is $T$ is reduced.
\end{prop}
\begin{proof} The first statement is a consequence of the theorem of cohomology and base change combined with Serre vanishing. The second statement is \cite[III, Thm.~9.9]{Hartshorne_AG}, or equivalently the fact that a coherent sheaf on a reduced scheme is locally free if and only if the fibre dimension is locally constant.
\end{proof}

\subsection{\texorpdfstring{The stack $\CCoh(X)$}{}}\label{sec:coh(X)}
We are ready to define the stack $\CCoh(X)$. We start by fixing some notation. Let $X$ and $T$ be schemes. We denote by 
\[
\CCoh(X)(T)
\]
the category whose objects are families of coherent sheaves on $X$ parametrised by $T$ according to \Cref{def:family-sheaves} (note that this includes properness of the support), and whose morphisms are isomorphisms of families (again as in \Cref{def:family-sheaves}). It is a (not full) subcategory of $\Coh(X\times T)$. By definition $\CCoh(X)(T)$ is a groupoid, namely a category in which all morphisms are isomorphisms.

\begin{remark} Let $X,T,S$ be schemes, and $f\colon S\to T$ a morphism. We denote by $f_X$ the induced morphism $(\id_X,f):X\times S\to X\times T$. If $\cF\in \CCoh(X)(T)$, then the pullback $f_X^*\cF\in \Coh(X\times S)$ is in $\CCoh(X)(S)$, because flatness is invariant under base change and so is properness. Also note that we implicitly use that the pullback of a coherent sheaf along a morphism of locally noetherian schemes is still coherent.
\end{remark}

\begin{definition}
Let $X$ be a scheme. The prestack defined by the association
\[
\begin{tikzcd}
    T \arrow[mapsto]{r} & \CCoh(X)(T)
\end{tikzcd}
\]
will be denoted $\CCoh(X)$ throughout. 
\end{definition}

Grothendieck \cite[Sec.~B]{Grothendieck_fpqc} proved that quasicoherent sheaves (and their morphisms) satisfy \emph{fpqc descent}. See also \cite[Thm.~4.23]{MR2223406} for another reference. This implies that $\CCoh(X)$ is in fact a \emph{stack}. 

We encourage the reader unfamiliar with stacks to skip the definition on first reading, and use instead the following consequences: for any schemes $X$ and $T$, we can identify $\CCoh(X)(T)$ as the category of morphism from $T$ to $\CCoh(X)$. A family $\cF$ is a morphism $T\to \CCoh(X)$, and an isomorphism of families is a $2$-arrow, or homotopy, between the two. If $f\colon S\to T$ is a morphism of schemes and $\cF\in \CCoh(X)(T)$, we view $f_X^\ast\cF$ as the composite with $f$ of the morphism $T\to \CCoh(X)$ defined by $\cF$.

We choose to work with the categories $\CCoh(X)(T)$ and not with isomorphism classes of families because, to do geometry, it is important to be able to define things locally and then glue them; it is well known that to glue sheaves defined on an open cover, it is not enough to require that they are isomorphic on pairwise intersections, we must choose isomorphisms and require a cocycle condition. Formalising this requirement leads directly to the definition of stack.

\begin{remark}\label{rmk:generalisations}
We need to mention that \emph{vast} generalisations of the definition (and properties) of $\CCoh(X)$ are available in the literature. For instance, in \cite[Thm.~2.1.1]{zbMATH05081797} the input is an algebraic stack $\CX \to B$ locally of finite presentation over an excellent algebraic space $B$, and the upshot is the existence and algebraicity of the stack of coherent sheaves with proper support on the fibres of $\CX \to B$. In \cite[Thm.~B]{zbMATH06477021}, the authors remove the hypothesis of being locally of finite presentation for the initial morphism: they prove that if $\CX \to \CS$ is a morphism of algebraic stacks with finite diagonal $\CX \to \CX \times_\CS\CX$, then the $\CS$-stack $\CCoh(\CX/\CS)$ is algebraic and has affine diagonal $\CCoh(\CX/\CS) \to \CCoh(\CX/\CS) \times_\CS \CCoh(\CX/\CS)$.
\end{remark}

\subsection{\texorpdfstring{Some open substacks of $\CCoh(X)$}{}}

Let $X$ be a scheme. In this subsection we define some natural open substacks of the stack $\CCoh(X)$. As for schemes, an open substack in a stack is determined by its points, that is, the isomorphism classes of morphisms $\Spec \bfk \to \CCoh(X)$. In other words, it can be specified by selecting a subcategory $\cA$ of $\Coh(X)$ closed under isomorphism.

The condition for such a subcategory $\CA \subset \Coh(X)$ to define an open substack is the following: for every scheme $T$, and every $\cF\in \CCoh(X)(T)$, the set $\set{t\in T\,|\, \cF_t\in \cA}$ is open in $T$.

This makes sense if we think of $\cF$ as a morphism $T\to \CCoh(X)$, since for any scheme $V$, a subset $A\subset V$ is open if and only if, for every morphism of schemes $f\colon T\to V$, the subset $f^{-1}(A)$ is open in $T$.

\subsubsection{Support in an open subset}
The first open substack of $\CCoh(X)$ we want to define is the stack $\CCoh(X)_U$ parametrising sheaves $\CF \in \CCoh(X)(\bfk)$ whose support lies in a fixed open subset $U\subset X$.

\begin{lemma}\label{lemma:supp-U}
Let $X$ be a scheme and $U\subset X$ an open subscheme. Let $\cA_U\subset \Coh(X)$ be the subcategory of sheaves $\CF$ such that $\Supp(\CF)\subset U$. Then $\cA_U$ defines an open substack $\CCoh(X)_U \subset \CCoh(X)$.
\end{lemma}
\begin{proof}
Let $T$ be a scheme and $\cF\in\CCoh(X)(T)$; let $Z$ be the closed subset of $X\times T$ defined by the closed subscheme $\Supp\cF\subset X\times T$. Let $C\subset X$ be the closed subset $X\setminus U$, and let $W\defeq Z\cap C\times T$. Then $W$ is closed in $Z$, and hence the projection $W\to T$ is universally closed; we denote by $Y$ its image, and by $A\subset T$ its complement, which is open in $T$. We leave it to the reader to verify that $\cF_t\in \cA_U$ if and only if $t\in A$.
\end{proof}

We now want to show that the open substack defined by $\cA_U$ is naturally isomorphic to $\CCoh(U)$; to do so, we first define a morphism of stacks $\CCoh(U)\to \CCoh(X)$ in a more general case.

\begin{prop}\label{embed}
Let $f\colon Y\to X$ be a morphism of schemes,  $T$ a scheme, and fix a family $\cF\in \CCoh(Y)(T)$; denote by $g$ the map $(f,\id_T)\colon Y\times T\to X\times T$. 
\begin{enumerate}
\item [\mylabel{embed-1}{(1)}] If $f$ is a closed embedding, then $g_*\cF\in \CCoh(X)(T)$;
\item [\mylabel{embed-2}{(2)}] If $f$ is an open embedding and $X$ is separated (e.g., quasiprojective), then $g_*\cF\in \CCoh(X)(T)$.
    \end{enumerate} 
\end{prop}
\begin{proof}
We proceed step by step.
\begin{enumerate}
\item If $f$ is a closed embedding, then so is $g$. Thus for every $(x,t)\in X\times T$ we have $(g_*\cF)_{(x,t)}=0$ if $x\notin Y$ and $(g_*\cF)_{(x,t)}=\cF_{(x,t)}$ if $x\in Y$, thus $g_*\cF$ is flat over $T$. Moreover $g$ is proper so $g_*\cF$ is coherent \cite[\href{https://stacks.math.columbia.edu/tag/02O3}{Tag 02O3}]{stacks-project}.
\item Let $Z\defeq\Supp \cF\subset Y\times T$, and let $W\subset X\times T$ the scheme theoretic image of $Z$. Since $Z$ is proper over $T$ and $X$ is separated, by \cite[Cor.~II.4.8(e)]{Hartshorne_AG}  the morphism $h=g|_Z\colon Z\to X\times T$ is proper, so the induced map $h=g|_Z\colon Z\to W$ is surjective. This implies that $W$ is also proper over $T$. Indeed, it is clearly separated and of finite type; it is also universally closed, since for every base change $T'\to T$ and induced cartesian diagram
\[
\begin{tikzcd}
Z'\ar[r]\ar[d, "\phi"] & Z\ar[d, "h"]\\
W'\ar[r]\ar[d, "\pi"] & W\ar[d]\\
T'\ar[r] & T
\end{tikzcd}
\]
and any $C'\subset W'$ closed subset, we have $\pi(C')=\pi(\phi(\phi^{-1}(C'))$ since $\phi$ is surjective (cf.~\cite[\href{https://stacks.math.columbia.edu/tag/01RY}{Tag 01RY}]{stacks-project}) and the latter is closed since $\pi\circ \phi$ is closed.

From the commutative diagram
\[
\begin{tikzcd}
Z \arrow[r, "h"] \arrow[d,hook, "i"]
& W \arrow[d,hook, "j"] \\
Y\times T \arrow[r, "g"]
& X\times T
\end{tikzcd}
\]
we get that 
\[
g_*\cF=g_*i_*i^*\cF=j_*h_*i^*\cF=(j\circ h)_\ast i^*\cF 
\]
hence $g_*\cF\in \Coh(X\times T)$, since $j\circ h$ is proper, and $\Supp(g_\ast \cF)\subset W$ is proper over $T$.\qedhere
\end{enumerate}
\end{proof}

\begin{corollary}\label{cor:Coh(U)-is-open}
Let $f\colon U\into X$ be an open embedding of schemes. Then the morphism $f_*\colon \CCoh(U)\to \CCoh(X)$ as defined in \Cref{embed} is an open embedding of stacks. 
\end{corollary}

See \cite[\href{https://stacks.math.columbia.edu/tag/0DLX}{Tag 0DLX}]{stacks-project} for the construction of the pushforward morphism 
\[
\begin{tikzcd}
\CCoh(Y) \arrow{r}{\pi_\ast} & \CCoh(X)
\end{tikzcd}
\]
for any quasifinite morphism $\pi \colon Y \to X$ of separated schemes.
    
\subsubsection{\texorpdfstring{The open substack of $m$-regular sheaves}{}}

\begin{definition} Let $X$ be a projective scheme with a chosen very ample line bundle $\OO_X(1)$. Let $\CF\in \Coh(X)$, and $m \in \BN$. We say that $\CF$ is {\em $m$-regular} in the sense of Castelnuovo--Mumford if $\HH^i(X,\CF(m-i))=0$ for all $i$ such that $0<i\le \dim X$. 
\end{definition} 

\begin{lemma}[{\cite{Mumford-Lectures-on-curves}}]\label{reg-lemma}
Every coherent sheaf $\CF$ on $X$ is $m$-regular for some $m \gg 0$. Moreover, if $\CF$ is $m$-regular, then
\begin{itemize}
    \item [\mylabel{mreg-i}{(i)}] $\CF$ is also $m'$-regular for $m' > m$,
    \item [\mylabel{mreg-ii}{(ii)}] $\CF(m)$ is globally generated, i.e.~the evaluation morphism 
    \[
    \begin{tikzcd}
    \HH^0(X,\CF(m)) \otimes_\bfk \OO_X(-m) \arrow{r} & \CF
    \end{tikzcd}
    \]
    is surjective, and 
    \item [\mylabel{mreg-iii}{(iii)}] $\HH^i(X,\CF(m))=0$ for all $i>0$.
\end{itemize}
\end{lemma}

\begin{remark}\label{rmk:m-reg-0dim}
Let $X$ be a projective scheme, and $\CF\in \Coh(X)$ a sheaf with $0$-dimensional support. Then $\CF$ is $m$-regular for every $m\in \BN$ \cite[p.~28]{modulisheaves}.
\end{remark}

\begin{prop}\label{prop:m-reg-is-open}
Let $(X,\OO_X(1))$ be a projective scheme, and $m\in \BN$. Then the subcategory of $\Coh(X)$ of $m$-regular sheaves defines an open substack $\CCoh_{m}(X)\subset \CCoh(X)$.
\end{prop}
\begin{proof}
Let $T$ be a scheme, $m\in \BN$ and $\cF\in \CCoh(X)(T)$. Set $d=\dim X$. Note that for every $n\in \BN$, one has $\cF(n)\in \CCoh(X)(T)$. Assume now that $p\in T$ is chosen in such a way that $\cF_p$ is $m$-regular. By applying cohomology and base change to $\cF(m-i)$ (for every $i$ with $0<i\le \dim X$) it follows that there exists $V_i$, open neighbourhood of $p$ in $T$, such that $\HH^i(X,\cF_q(m-i))=0$ for every $q\in V_i$. Let $V\defeq V_1\cap\cdots\cap V_d$; then $V$ is an open neighbourhood of $p$ in $T$, and for every $q\in V$ the sheaf $\cF_q$ is $m$-regular.
\end{proof}

 \begin{corollary}
     Let $(X,\OO_X(1))$ be a projective scheme, $T$ a scheme, and $\cF\in \CCoh(X)(T)$. Then there exists $m_0\in \BN$ such that for every $p\in T$ the sheaf $\cF_p$ is $m_0$-regular. 
 \end{corollary}

\begin{proof}
Each subset
\[
T_m=\set{p\in T\,|\, \cF_p\text{ is $m$-regular}}\subset T.
\]
is open in $T$ by \Cref{prop:m-reg-is-open}. Moreover, there is an inclusion $T_m\subset T_{m+1}$ by \Cref{reg-lemma}, and $T=\bigcup_{m\in \BN}T_m$. Since $T$ is quasicompact, there must be an $m_0$ such that $T$ is the union of all $T_m$ for $m\le m_0$, which implies $T=T_{m_0}$.
\end{proof}
\subsubsection{Fixed Hilbert polynomial}\label{sec:fixed-P}
The final open substack of $\CCoh(X)$ we want to define is the one parametrising sheaves with fixed Hilbert polynomial. Again we have to assume $X$ is projective, and fix a very ample invertible sheaf $\OO_X(1)$. 

For each $\CF\in \Coh(X)$ its Hilbert polynomial $P_\CF\in \BQ[z]$, with respect to $\OO_X(1)$, is defined by setting $P_\CF(k)=\chi(\CF(k))$ for every $k\in \BZ$; equivalently, $P_\CF$ can be defined by requiring $P_\CF(k)=h^0(X,\CF(k))$ for every $k$ such that $\HH^i(X,\CF(k))=0$ for all $i~>~0$, since every sufficiently large $k$ satisfies this condition by Serre vanishing. We denote by $\Coh^P(X)$ the subcategory of $\Coh(X)$ whose objects are sheaves with Hilbert polynomial $P$.

The goal of this subsection is to show that $\Coh^P(X)$ defines an open and closed substack $\CCoh^P(X)$ of $\CCoh(X)$. By definition, it is enough to prove the following.

\begin{prop}\label{prop:hilb-poly-constant}
Let $(X,\OO_X(1))$ be a projective scheme, $T$ a connected scheme and $\cF\in \CCoh(X)(T)$. Then the Hilbert polynomial $P_t$ of $\cF_t$ is the same for every $t\in T$.
\end{prop}
\begin{proof}
Choose $m_0\in \mathbb N$ such that $\cF_t$ is $m_0$-regular for every $t\in T$. Since $\cF_t$ is also $m$-regular for every $m\ge m_0$, it follows that $\HH^i(X,\cF_t(m))=0$ for every $i>0$ and every $m\ge m_0$. Let $\pr_T \colon X \times T \to T$ be the projection. By cohomology and base change, for every $m\ge m_0$ the sheaf $\pr_{T\ast}\cF(m)$ is locally free on $T$ and commutes with base change; its rank $r_m$ (which is constant since $T$ is connected) is equal to the dimension of its fibre  $\pr_{T\ast}\cF(m)_t\otimes_{\OO_T} \kappa(t)$  for every $t\in T$, and  
\[
\begin{tikzcd}
\pr_{T\ast}\cF(m)_t\otimes_{\OO_T} \kappa(t)\arrow{r} & \HH^0(X,\cF_t(m))
\end{tikzcd}
\]
is an isomorphism. Hence for every $t\in T$, one has $P_t(m)=r_m$ for every $m\ge m_0$. Given $t_1,t_2\in T$, the polynomial $P_{t_1}-P_{t_2}$ must vanish for every $m\ge m_0$, whence the result.
\end{proof}

It follows that there is an open substack 
\[
\CCoh^P(X)\subset \CCoh(X)
\]
parametrising coherent sheaves $\CF$ with Hilbert polynomial $P$. Since $\CCoh(X)$ is the disjoint union, varying $P$, of these open substacks, it follows that each of them is, in fact, open and closed.

\subsection{\texorpdfstring{Exhaustion of $\CCoh^P(X)$ by global quotients}{}}

Let $(X,\OO_X(1))$ be a polarised projective scheme, $P \in \BQ[z]$ a polynomial.
Since every coherent sheaf is $m$-regular for some $m$, we have an exhaustion of $\CCoh^P(X)$ by open substacks
\[
\cdots \into \CCoh^P_m(X) \into \CCoh^P_{m+1}(X) \into \cdots
\]
and it turns out that each $\CCoh^P_m(X) = \CCoh^P(X) \cap \CCoh_m(X)$ is a global quotient stack. More precisely, consider the open subscheme
\[
U_m^P\subset \Quot_{X}(\OO_{X}(-m)^{\oplus P(m)},P)
\] 
parametrising isomorphism classes $[\theta]$ of quotients
\[
\begin{tikzcd}
\OO_{X}(-m)^{\oplus P(m)} \arrow[two heads]{r}{\theta} & \CF
\end{tikzcd}
\]
such that $\CF$ is $m$-regular and the induced map
\begin{equation}\label{iso-twisted}
\begin{tikzcd}
    \bfk^{\oplus P(m)} = \HH^0(X,\OO_X^{\oplus P(m)}) \arrow{rr}{\HH^0(X,\theta(m))} && \HH^0(X,\CF(m))
\end{tikzcd}
\end{equation}
is an isomorphism. Then we have the following result.

\begin{theorem}\label{thm:global-quot-1}
Let $(X,\OO_X(1))$ be a polarised projective scheme, $P \in \BQ[z]$ a polynomial, $m$ an integer. There is an isomorphism of stacks
\[
\CCoh^P_m(X) \cong [U_m^P / \GL_{P(m)}].
\]
In particular, $\CCoh^P_m(X)$ is an algebraic stack, with $U^P_m$ serving as a smooth atlas.
\end{theorem}

\begin{proof}
Fix a point $[\CF] \in \CCoh^P_m(X)$. By \Cref{reg-lemma}\ref{mreg-iii}, we have $P(m) = h^0(X,\CF(m))$. By \Cref{reg-lemma}\ref{mreg-ii}, we have a surjection $\OO_X(-m)^{\oplus P(m)} \onto \CF$, and a choice of isomorphism $\bfk^{\oplus P(m)} \cong \HH^0(X,\CF(m))$ allows one to realise $\CF$ as a quotient appearing as a point $[\theta]$ of $U^P_m$. The scheme $U^P_m$ is naturally acted on by $\GL_{P(m)}$. Moving the point $[\theta]$ within its $\GL_{P(m)}$-orbit corresponds precisely to moving $\CF$ into its isomorphism class. This establishes an identification of the two stacks at the level of $\bfk$-points.

This identification naturally extends to arbitrary families. To see this, let $T$ be a scheme, $\cF\in \Coh(X\times T)$ a $T$-flat family of $m$-regular coherent sheaves with Hilbert polynomial $P$, and let $\pr_T \colon X \times T \to T$ be the projection. There is a $\GL_{P(m)}$-torsor
\[
\Isom(\OO_T^{P(m)},\pr_{T\ast}\cF(m)) \to T
\]
equipped with a natural $\GL_{P(m)}$-invariant morphism into $U^P_m$ (this uses that any base change of the evaluation map $\pr_T^\ast \pr_{T\ast}\cF(m) \to \cF(m)$ is surjective), which defines a $T$-valued point of $[U_m^P / \GL_{P(m)}]$. This construction yields an equivalence (compatible with pullbacks) between the $T$-valued points of the stacks involved, proving the sought after isomorphism.
\end{proof}

We conclude that $\CCoh^P(X)$ has an exhaustion by open substacks which are global quotient stacks. We shall see in \Cref{prop:Coh^n-global-quotient} that, if $P\equiv n$ is a constant polynomial, then $\CCoh^n(X)$ is itself a global quotient.

\section{The stack of 0-dimensional sheaves}\label{sec:0-dim-sheaves}
In this section we study the open substack $\CCoh^n(X)\subset \CCoh(X)$ of 0-dimensional sheaves of fixed length $n\geq 0$. If $X$ is projective, our discussion in \Cref{sec:fixed-P} gives the definition of $\CCoh^n(X)$ as an open substack of $\CCoh(X)$ by setting $P=n$, i.e.~by taking a constant Hilbert polynomial. We nevertheless provide full details in the general setup in the next subsection. In particular, this sheds light on the importance of requiring the support to be proper.

We start with an easy but fundamental lemma, which shows that in the $0$-dimensional case, the stack $\CCoh^n(X)$ parametrizes all coherent sheaves on $X$ even if $X$ is not projective.

\begin{lemma}\label{basic-lemma}
Let $\CF$ be a coherent sheaf on a quasiprojective scheme $X$ such that $\Supp(\CF)$ is $0$-dimensional. Then \begin{enumerate}
\item  The set-theoretic support $\Supp_{\mathrm{set}}(\CF)$ is a finite subset of $X$.
\item Let $Z=\Supp(\CF)$. Then $Z$ is a finite disjoint union of closed subschemes
$\{\iota_j \colon Z_j \into X\}_{j\in J}$ with $Z_{j,\red}$ equal to a $\bfk$-point $p_j$ in $X$.
\item The support $\Supp(\CF)$ is proper, hence $\CF$ defines a $\bfk$-point of $\CCoh(X)$.
\item The natural morphism $\CF\to \bigoplus_{j \in J} \iota_{j*}\CF|_{Z_j}$ is an isomorphism.
    \end{enumerate}
\end{lemma}
\begin{proof}
We proceed step by step.
\begin{enumerate}
\item The set-theoretic support of $\CF$ is closed in $X$ and $0$-dimensional, hence finite since $X$ is noetherian.
\item The underlying set of $Z$ is a union of finitely many $\bfk$-points $\{p_j\}_{\{j\in J\}}$, we can take $Z_j \defeq Z\cap U_j$ where $U_j$ is the open subscheme of $X$ obtained by removing all $p_i$ with $i\in J\setminus\{j\}$.
\item A finite set of $\bfk$-points is proper.
\item This can be checked at stalks.
\end{enumerate}
The proof is complete.
\end{proof}

\begin{corollary}\label{cor:glob-gen}
 Let $\CF$ be a coherent sheaf on a quasiprojective scheme $X$ such that $\Supp(\CF)$ is $0$-dimensional. Then the evaluation map $\HH^0(X,\CF)\otimes_{\bfk}\OO_X\to \CF$ is surjective. 
\end{corollary}
\begin{proof}
    This is true for every coherent sheaf if $X$ is affine. By \Cref{basic-lemma}, it is enough to prove it for $\CF$ having set theoretic support in one point $p$. Let $U$ be an open affine in $X$ containing $p$. Then $\HH^0(X,\CF)\to \HH^0(U,\CF|_U)$ is an isomorphism, and we reduce to the affine case.
\end{proof}
\subsection{The definition of \texorpdfstring{$\CCoh^n(X)$}{}}
Let $X$ be a quasiprojective scheme, and let $n \geq 0$ be an integer. In this subsection we introduce $\CCoh^n(X)$, the stack parametrising sheaves $\CF \in \Coh(X)$ such that 
\[
\dim \CF=0, \quad \chi(\CF)=n. 
\]
The definition goes as follows. Let $T$ be a scheme, $\cF\in\CCoh(X)(T)$. The subset
\begin{equation}\label{eqn:open-dim}
\Set{t \in T | \dim \cF_t < i}\subset T    
\end{equation}
is open for all $i>0$, by upper-semicontinuity of the function
\begin{equation}\label{dim-supp-map}
\begin{tikzcd}
t \arrow[mapsto]{r} & \dim \cF_t.
\end{tikzcd}
\end{equation}
To check upper-semicontinuity of \eqref{dim-supp-map}, it is enough to observe that by \cite[\href{https://stacks.math.columbia.edu/tag/056H}{Tag 056H}]{stacks-project} one has, at the level of topological spaces, the identity
\[
\lvert \Supp(\cF) \cap (X\times \set{t}) \rvert = \lvert \Supp(\cF_t) \rvert,
\]
Knowing this, the semi-continuity statement follows from semi-continuity of fibre dimension along a closed morphism locally of finite type \cite[Cor.~14.113]{zbMATH05585185}, e.g.~a proper morphism such as $\Supp(\cF) \to T$.

It follows from openness of \eqref{eqn:open-dim} that there is an open substack
\[
\CCoh(X)_0\subset \CCoh(X)
\]
parametrising sheaves with 0-dimensional support (corresponding to $i=1$ in \eqref{eqn:open-dim}). In turn, the substack $\CCoh(X)_0$ admits a decomposition 
\[
\CCoh(X)_0 = \coprod_{n\geq 0} \CCoh^n(X)
\]
where $\CCoh^n(X)$ are open (and closed) substacks. Indeed, if $\cF\in\CCoh(X)_0(T)$, fixing $n$ amounts to fixing the rank of the locally free sheaf 
\[
\pr_{T\ast}\cF
\]
over $T$, and the rank is a locally constant function $T\to \BN$, formally defined by sending 
\[
\begin{tikzcd}
t \arrow[mapsto]{r} & \dim_{\kappa(t)} \pr_{T\ast}\cF \otimes_{\OO_T} \kappa(t),
\end{tikzcd}
\]
therefore it is continuous (having endowed $\BN$ with the discrete topology), and hence its fibres are open (and closed). The open and closed stratification by the rank of the pushforward defines the stacks $\CCoh^n(X)$, which are by construction open and closed in $\CCoh(X)_0$, and in particular open in $\CCoh(X)$.

\begin{example}\label{ex:few-points}
Let $X$ be a $\bfk$-variety. We have
\[
\CCoh^0(X)=\Spec \bfk, \qquad \CCoh^1(X) = X \times \mathrm B\BG_{m}.
\]
In particular, $\CCoh^1(X)$ is smooth as soon as $X$ is. If $n=2$, there are two possibilities for a closed point $[\CF] \in \CCoh^2(X)$. Namely, either
\begin{itemize}
    \item $\CF=\OO_Z$ for a length 2 closed subscheme $Z \into X$, or
    \item $\CF=\OO_p^{\oplus 2}$ for a closed point $p\in X$.
\end{itemize}
See \cite{MR18} for a geometric description of the stacks $\CCoh^n(\BA^2)_0$, for $n \leq 4$, in terms of the stratification by minimal number of generators for finite length $\bfk[x,y]$-modules.
\end{example}

\subsection{\texorpdfstring{Description of $\CCoh^n(X)$ as global quotient}{}}
In this subsection we realise $\CCoh^n(X)$ as a global quotient stack, whenever $X$ is a projective $\bfk$-variety. We need the Quot scheme of points as a key tool. If $\CE \in \Coh(X)$ is a coherent sheaf, the Quot scheme of points $\Quot_X(\CE,n)$ is the $\bfk$-scheme parametrising isomorphism classes $[\CE \onto \CF]$ of surjections in $\Coh(X)$, where $[\CF] \in \CCoh^n(X)$ (cf.~\Cref{sec:quot}).

\begin{lemma}\label{lemma:quot-to-coh}
Let $X$ be a scheme. Fix $\CE \in \Coh(X)$. For every $n\geq 0$, there is a morphism 
\begin{equation}\label{rho-morphism}
\begin{tikzcd}
\rho_{\CE,n} \colon \Quot_X(\CE,n) \arrow{r} & \CCoh^n(X)
\end{tikzcd}
\end{equation}
sending a point $[\CE \onto \CF]$ of the Quot scheme to the point $[\CF]$.

If $\CE$ is locally free of rank $r>0$, then $\rho_{\CE,n}$ is smooth of relative dimension $rn$.
\end{lemma}

\begin{proof}
The morphism $\rho_{\CE,n}$ is constructed by sending the universal exact sequence
\[
0 \to \mathscr S \to \pr_X^\ast \CE \to \mathscr F \to 0,
\] 
where $\pr_X$ denotes the first projection $X \times \Quot_X(\CE,n)\to X$, to the object 
\[
\cF \in \CCoh^n(X)(\Quot_X(\CE,n)).
\]
Let us denote by $p$ a $\bfk$-valued point in $\Quot_X(\CE,n)$ defined by an exact sequence 
\[
0 \to S \to  \CE \to F \to 0.
\] 
The relative tangent and obstruction spaces to the morphism $\rho_{\CE,n}$ at $p$ are $\Hom(\CE,F)$ and $\Ext^1(\CE,F)$, respectively (since deformations of a surjective morphism of coherent sheaves are surjective). Since $\CE$ is locally free, $\Ext^1(\CE,F)=\HH^1(\CE^*\otimes F)$, which vanishes since $\Supp F$ has dimension $0$, thus the morphism $\rho_{\CE,n}$ is smooth at $p$. The relative tangent space is $\Hom(\CE,F)=\HH^0(X,\CE^*\otimes F)$, which is the direct sum of $\HH^0(X,\CE^*\otimes F|_{Z_j})$ using the notation in Lemma \ref{basic-lemma}. For each $j\in J$, there exists an open subset $U_j$ of $X$ containing $p_j$, and thus $Z_j$, such that $\CE|_{U_j}$ is trivial. Hence $\HH^0(X,\CE^*\otimes F|_{Z_j})$ is isomorphic to  $\HH^0(X,F^{\oplus r}|_{Z_j})$. Summing over $j\in J$, we get that $\Hom(\CE,F)$ is isomorphic to $\HH^0(X,F^{\oplus r})=\HH^0(X,F)^{\oplus r}$ hence has dimension $nr$. This is then the relative dimension of $\rho_{\CE,n}$.
\end{proof}
Consider now the Quot scheme
\[
\mathrm Q_n = \Quot_X(\OO_X^{\oplus n},n).
\]
By \Cref{lemma:quot-to-coh}, there is a smooth morphism
\[
\begin{tikzcd}
\mathrm Q_n \arrow{r}{\rho_{n}} & \CCoh^n(X)
\end{tikzcd}
\]
forgetting the surjection and only retaining the quotient sheaf. 
In fact, $\rho_n$ is surjective by \Cref{cor:glob-gen} (see also the proof of \cite[Lemma 3.2]{d-critical_quot}).

Let 
\[
\begin{tikzcd}
  \OO_{X \times \mathrm Q_n}^{\oplus n} \arrow[two heads]{r} & \cF
\end{tikzcd}
\]
be the universal quotient, and denote by $\pi\colon X \times \mathrm Q_n \to \mathrm Q_n$ the projection. Note that by \Cref{flatness-criteria} the pushforward $\pi_\ast \cF$ is locally free of rank $n$. The locus
\[
A_{X,n} \subset \mathrm Q_n 
\]
where the composition
\[
\begin{tikzcd}\OO_{\mathrm Q_n}^{\oplus n}\arrow{r} & \pi_\ast \OO_{X \times \mathrm Q_n}^{\oplus n}\arrow{r} & \pi_\ast \cF
\end{tikzcd}
\]
is an isomorphism is open in $Q_n$ and $\GL_{n}$-invariant.

Informally speaking, a point in $A_{X,n}$ corresponds to a linear isomorphism $\bfk^{\oplus n}\to \HH^0(X,\CF)$ with $\CF$ a $\bfk$-point of $\CCoh^n(X)$, hence to a choice of basis in $\HH^0(X,\CF)$; this implies that the natural morphism $A_{X,n}\to \CCoh^n(X)$ is a principal $\GL_n$-bundle whose fibres are all possible bases of $\HH^0(X,\CF)$.

This intuitive idea can be formalised as follows.

\begin{theorem}\label{prop:Coh^n-global-quotient}
Let $(X,\OO_X(1))$ be a projective variety, $n \geq 0$ an integer. There is an isomorphism of stacks
\[
\begin{tikzcd}
    {[}A_{X,n}/\GL_{n}{]} \arrow{r}{\sim} & \CCoh^n(X).
\end{tikzcd}
\]  
In particular, $\CCoh^n(X)$ is algebraic and a global quotient, with $A_{X,n}$ serving as a smooth atlas.
\end{theorem}

\begin{proof}
The condition defining $A_{X,n}$ defines an isomorphism $\bfk^{\oplus n} \simto \HH^0(X,\CF)$ for every $[\CF] \in \CCoh^n(X)$. The proof is then identical to the proof of \Cref{thm:global-quot-1}, with $P\equiv n$ (and noting that $0$-dimensional sheaves are $m$-regular for all $m$, cf.~\Cref{rmk:m-reg-0dim}).
\end{proof}

For instance, one has $A_{X,1} = \Quot_X(\OO_X,1)=X$, so that
\[
\CCoh^1(X) \cong [\Quot_X(\OO_X,1)/\BG_{m}] = [X/\BG_{m}]=X \times \mathrm B\BG_{m},
\]
where the last identity follows by triviality of the $\BG_m$-action. This recovers the $n=1$ part of \Cref{ex:few-points}.

\subsection{\texorpdfstring{Commuting matrices and the stack $\CCoh^n(\BA^d)$}{}}\label{sec:Coh(A^d)}
In this subsection we prove a folklore description of the stack of $0$-dimensional coherent sheaves of finite length on $\BA^d$, the $d$-dimensional affine space over a field $\bfk$. It is a global quotient stack, encoding linear algebraic data ($d$ pairwise commuting matrices up to conjugation). See also \cite[App.~A]{Joachim-Sivic}.

\begin{theorem}\label{thm:affine-coh^n}
Fix two integers $d>0$ and $n\geq 0$. There is a natural isomorphism
\[
\begin{tikzcd}
\CCoh^n(\Ab^d) \arrow{r}{\sim} & {[}C(n,d)/\GL_{n}{]},
\end{tikzcd}
\]
where $C(n,d)\subset \End_\bfk(\bfk^n)^{\oplus d}$ is the closed subscheme of $d$-tuples of pairwise commuting $n\times n$ matrices, 
\[
C(n,d)=
\Set{(A_1,\ldots,A_d)\in \End_\bfk(\bfk^n)^{\oplus d}\,|\, A_iA_j=A_jA_i \mbox{ for all }1 \leq i<j\leq d},
\]
and the linear group $\GL_{n}$ acts by simultaneous conjugation 
\[
g\cdot (A_1,\ldots,A_d)
=(gA_1g^{-1},\ldots, gA_dg^{-1}).
\]
In particular, 
\[
\CCoh^n(\BA^1) \cong [\End_\bfk(\bfk^n)/\GL_{n}]
\]
is a smooth algebraic stack.
\end{theorem}

\begin{proof}
Let $M(n,d)$ denote the scheme $\BA^{dn^2}$, viewed as $d$-tuples of $n\times n$ matrices. Let $S=\bfk\langle x_1,\ldots,x_d\rangle$ be the free $\bfk$-algebra on $d$ generators. The stack quotient 
\[
[M(n,d)/\GL_n]
\]
is the stack of $n$-dimensional (left) $S$-modules. It contains, as a closed substack, the stack of $n$-dimensional $\bfk[x_1,\ldots,x_d]$-modules, namely $\CCoh^n(\BA^d)$. To see this, note that $[M(n,d)/\GL_n]$ can be interpreted as the stack $\Rep_n(Q_d)$ of $n$-dimensional representations of the $d$-loop quiver $Q_d$. One can also consider the quiver with relations $(Q_d,R_d)$, where the relations $R_d$ are the commutator relations $[x_i,x_j] = 0$ for $1\leq i<j\leq d$, where $x_1,\ldots,x_d$ are the $d$ loops. The corresponding stack $\Rep_n(Q_d,R_d)$ of $n$-dimensional representations is exactly the stack of $S/R_d$-modules of length $n$, namely $\CCoh^n(\BA^d)$. But the relations can be imposed in the prequotient, giving rise to the sought after isomorphism. 

Let us make the isomorphism explicit. Let $U$ be a $\bfk$-scheme. We start by the following observation: a morphism $U \to \CCoh^n(\BA^d)$, namely an object in the groupoid $\CCoh^n(\BA^d)(U)$, is the same as an $\OO_U$-algebra morphism $\OO_U[x_1,\ldots,x_d] \to \lEnd_{\OO_U}(\OO_U^{\oplus n})$.

Let us suppose we are given an object 
\begin{equation}\label{eqn:quotient-stack-obj}
\begin{tikzcd}
P \arrow{r}{f}\arrow[swap]{d}{t} & C(n,d) \\
U &
\end{tikzcd}
\end{equation}
of $[C(n,d) / \GL_{n}](U)$, which means that $t$ is a principal $\GL_{n}$-bundle and $f$ is a $\GL_{n}$-equivariant morphism. This is, locally on $U$, the data
\begin{equation}\label{local-gl-n-bundle}
\begin{tikzcd}
\Spec A \times \GL_{n} \arrow{r}{f_A}\arrow[swap]{d}{t_A} & C(n,d) \\
\Spec A &
\end{tikzcd}
\end{equation}
where $t_A$ is now the first projection, and $f_A$ satisfies $f_A(y,gh) = h\cdot f_A(y,g)$ for all $g,h \in \GL_{n}$ and $y \in \Spec A \subset U$. Each map $f_A$ determines an $A$-algebra homomorphism
\[
\begin{tikzcd}
\widetilde{f}_A \colon A[x_1,\ldots,x_d] \arrow{r} & \End_A(A^{\oplus n}).
\end{tikzcd}
\]
The fact that the diagrams \eqref{local-gl-n-bundle} glue to \eqref{eqn:quotient-stack-obj} corresponds to the fact these algebra morphisms glue to a global $\OO_U$-algebra morphism
\[
\OO_U[x_1,\ldots,x_d] \to \lEnd_{\OO_U}(\OO_U^{\oplus n}).
\]
This yields a $U$-valued point of $\CCoh^n(\BA^d)$.
It is clear by the construction that the procedure can be inverted, and that isomorphisms of objects as in \eqref{eqn:quotient-stack-obj} correspond to isomorphisms in the groupoid $\CCoh^n(\BA^d)(U)$.
\end{proof}

\begin{remark}
We would like to point out a connection with an area of current active research, namely the case $d=3$. This is somewhat special, in that the commuting variety $C(n,3)$ is a \emph{critical locus}. In this case, $\CCoh^n(\BA^3)$ is a `critical stack', namely the zero locus of an exact 1-form on the quotient stack $[\End_\bfk(\bfk^n)^{\oplus 3} / \GL_{n}]$. Indeed, the commuting variety $C(n,3)$ is a scheme-theoretic critical locus, i.e.~it can be written as $Z(\dd f_n) = \set{\dd f_n = 0}$ for a very explicit regular function $f_n$ on a smooth $\bfk$-variety $Y_n$. The $\bfk$-variety $Y_n$ is the affine space $\End_{\bfk}(\bfk^n)^{\oplus 3}$, and $f_n$ is the function sending $(A_1,A_2,A_3) \mapsto \Tr A_1[A_2,A_3]$. The key observation is that the simultaneous vanishing $[A_i,A_j]=0$ is equivalent to the single relation $\dd f_n = 0$. This fact is relevant, and ubiquitous, in Donaldson--Thomas theory. See \cite{d-critical_quot} for a careful analysis of the identification
\[
\CCoh^n(\BA^3) \cong  [Z(\dd f_n) / \GL_n],
\]
as well as a derived upgrade of this isomorphism and a comparison of the natural d-critical structures on each of these spaces.
\end{remark}

\begin{remark}
Let $Y \subset \BA^d$ be an affine variety, cut out by an ideal $J = (f_1,\ldots,f_r)$ in  $\bfk[x_1,\ldots,x_d]$. With minimal modifications, one can adapt the proof of \Cref{thm:affine-coh^n} to prove that
\[
\CCoh^n(Y) \cong [C(n,Y) / \GL_n],
\]
where $C(n,Y) \subset C(n,d)$ is the closed subscheme of $d$-tuples $(A_1,\ldots,A_d)$ of pairwise commuting $n\times n$ matrices cut out by the relations
\[
f_i(A_1,\ldots,A_d) = 0,\qquad i=1,\ldots,r.
\]
The $\GL_n$-action is, once again, simultaneous conjugation.
\end{remark}

\section{The naive Coh-to-Chow morphism}\label{sec:naive-coh-to-chow}

The most naive way to parametrise `$n$ points' on a quasiprojective scheme $X$ is by looking at the $n$-fold power $X^n$ of $X$. Getting rid of the ordering produces the \emph{$n$-th symmetric product}, namely the GIT quotient $\Sym^n(X) = X^n/\FS_n$ described in \Cref{sec:sym^n(X)}, where $\FS_n$ is the symmetric group on $n$ letters, acting on $X^n$ in the natural way. A sheaf $\CF\in\Coh(X)$ with finite support and $\chi(\CF)=n$ yields a natural element $\cycle(\CF) \in \Sym^n(X)$, namely 
\[
\cycle(\CF) = \sum_{x \in X}\length_{\OO_{X,x}}\CF_x\cdot x.
\]
It is a deep fact that the association $[\CF] \mapsto \cycle (\CF)$ is the set-theoretic part of a natural \emph{morphism}
\begin{equation}\label{supp}
\begin{tikzcd}
    \supp_X^n \colon \CCoh^n(X) \arrow{r} & \Sym^n(X).
\end{tikzcd}
\end{equation}
In this section we construct (assuming $\bfk$ has characteristic $0$) a `naive' version of $\supp_X^n$, namely the morphism
\begin{equation}\label{naive-supp-map}
\begin{tikzcd}
 \CCoh^n(X)_{\red} \arrow{r} & \Sym^n(X).
\end{tikzcd}
\end{equation}
We refer to Rydh \cite[IV, Prop.~7.8]{Rydh1} for full details on the construction of \eqref{supp} in the most general setup, and to our companion paper \cite{Coh(X)-motivic} for a simplified version of this construction, still carefully following Rydh's approach.

\subsection{The symmetric product}\label{sec:sym^n(X)}

Let $n\ge 1$ and let $X=\Spec R$ be an affine scheme; the product $X^n$ of $n$ copies of $X$ is also affine and naturally isomorphic to $\Spec R^{\otimes n}$. There is a natural action of the symmetric group $\FS_n$ on $X^n$ by permutation, and thus on $R^{\otimes n}$; the $n$-th symmetric product $\Sym^n(X)$ is defined to be the affine scheme $\Spec ((R^{\otimes n})^{S_n})$. The inclusion $(R^{\otimes n})^{S_n}\into R^{\otimes n}$ induces a natural morphism $X^n\to \Sym^nX$, which is a basic example of a GIT (universal good) quotient (cf.~\cite[Thm.~1.1]{Mumford-GIT}).

The construction commutes in a natural way with localisations, and thus for any scheme $X$, affine or not, we have a natural affine morphism $X^n\to \Sym^n(X)$ with the universal property that, for any scheme $S$, it induces a bijection from $\Mor(S,\Sym^nX)$ to the set of $\FS_n$-invariant morphisms $X^n\to S$. This morphism $X^n \to \Sym^n(X)$ is, once again, a universal good quotient in the sense of geometric invariant theory.

The construction also commutes with closed and open embeddings, that is if $X\to Y$ is a locally closed embedding, there is a natural cartesian diagram 
\[
\begin{tikzcd}
X^n\MySymb{dr}\arrow[hook]{r}\arrow{d} & Y^n \arrow{d} \\
\Sym^n(X)\arrow[hook]{r} & \Sym^n(Y)
\end{tikzcd}
\]
where $\Sym^n(X) \into \Sym^n(Y)$ is also a locally closed embedding.

\subsection{The morphism to Chow}\label{sec:Coh-Chow-naive}

Let $X$ be a scheme, and $[\CF]\in \CCoh^n(X)$. We can write 
\[
\CF = \bigoplus_{p\in \Supp(\CF)} \CF|_{U_p},
\]
where $U_p$ is any open subscheme of $X$ such that $U_p\cap \Supp(\CF)=\set{p}$. Note that the restriction map $\CF(U_p)\to \CF_p$ is an isomorphism.

We can associate to $\CF$ the element 
\[
\cycle(\CF) = \sum_{x \in \Supp(\CF)}\length_{\OO_{X,x}}\CF_x\cdot x \in \Sym^n(X).
\]
Note that, by \cite[\href{https://stacks.math.columbia.edu/tag/00IU}{Tag 00IU}]{stacks-project}, this equals the $0$-cycle
\[
\sum_{x\in\Supp(\CF)} \dim_{\bfk}\CF_x\cdot x.
\]
A key result is that the set-theoretic association $[\CF]\mapsto \cycle(\CF)$ is, in fact, a morphism of stacks.

\begin{theorem} There is a natural morphism $\supp_X^n\colon \CCoh^n(X)\to \Sym^nX$ which sends the point $[\CF]$ to $\cycle(\CF)$ and such that, for any open subset $U\subset X$, the induced diagram
\[
\begin{tikzcd}
    \CCoh^n(U)\ar[r]\ar[d]& \CCoh^n(X)\ar[d]\\
    \Sym^n(U)\ar[r] &\Sym^n(X)
\end{tikzcd}
\]
is cartesian.
\end{theorem}
A complete construction of $\supp_X^n$ can be found in Rydh's work Rydh \cite[IV, Prop.~7.8]{Rydh1}. We choose not to include it, but we make comments instead.

First of all, to show that the diagram above is cartesian, it is enough to show that it commutes, since by \Cref{cor:Coh(U)-is-open} we know that $\CCoh^n(U)\to \CCoh^n(X)$ is an open embedding with image 
\[
\CCoh^n(X)_U = \CCoh^n(X) \times_{\Sym^n(X)}\Sym^n(U),
\]
namely the open substack inverse image of $\Sym^n(U)$ for any morphism $\CCoh^n(X)\to\Sym^n(X)$ mapping $[\CF]$ to $\cycle(\CF)$.

It is thus enough to construct the support morphism when $X$ is an affine scheme and show that it commutes with localisation. If $X=\Spec R$ is affine, so is $\Sym^n(X)\defeq\Spec A$ with $A=(R^{\otimes n})^{\FS_n}$; we also note that, if $\bfk$ has characteristic $0$, then  $A$ is generated by elements of the form $f^{\otimes n}$ for $f\in R$. 

To construct $\supp_X^n$, one would need to construct, for any scheme $T$ and for any family $\cF\in\CCoh^n(X)(T)$, a morphism $T\to \Sym^n(X)$ compatible with pullback. Since $\Sym^n(X)$ is affine, this is equivalent to giving a homomorphism of $\bfk$-algebras $A\to \OO_T(T)$.

We can define such a map directly on the generators by mapping 
\[
\begin{tikzcd}
    f^{\otimes n}\arrow[mapsto]{r} & \det q_*f.
\end{tikzcd}
\]
In a little more detail, we view multiplication by $f$ as an endomorphism $f\colon \cF\to\cF$. Pushing forward via the projection $q\colon X\times T\to T$, we get an induced homomorphism $q_*f\colon q_*\cF\to q_*\cF$ of locally free sheaves of rank $n$ on $T$, and taking determinants gives us a homomorphism  
\[
\begin{tikzcd}
\det q_*f\colon \det q_*\cF\arrow{r} & \det q_*\cF.
\end{tikzcd}
\]
Finally, since  $\cL\defeq\det q_*\cF$ is an invertible sheaf, we have 
\[
\Hom(\cL,\cL)=\Hom(\OO_T,\cL^*\otimes\cL)=\OO_T(T).
\]

This construction clearly commutes with base change, since pushforward to the base of the family does, and taking determinants commutes with pullback. This would complete the proof if we could show that there is a (necessarily unique) $\bfk$-algebra homomorphism $A\to \OO_T(T)$ with the given value on every $f^{\otimes n}$, but we will not do that.

Instead, we will note  that this does define the correct map when $T=\Spec \bfk$; this in turn yields the definition of the  morphism $\CCoh^n(X)_{\red} \to \Sym(X)$, announced in \eqref{naive-supp-map}, with the required property. This is already a very useful result, since it means that any stratification (decomposition in locally closed subsets, typically induced by partitions --- as in \Cref{sec:stratification} below --- or by a stratification in $X$) on $\Sym^nX$ induces one on $\CCoh^n(X)$, and thus on every family of coherent sheaves of length $n$, compatible with base change.

In particular, we get induced morphisms $\Quot_X(\CE,n)_{\red}\to \Sym^n(X)$ for any fixed $\CE\in \Coh(X)$, a special case of so called Quot-to-Chow morphisms.

\subsection{Coh-to-Chow for the affine space}\label{sec:coh-to-chow-affine-and-punctual}
The Coh-to-Chow morphism
\[
\CCoh^n(\BA^d)_{\red} \to \Sym^n(\BA^d)
\]
for $X=\BA^d$ is particularly easy to describe, due to the description of \Cref{thm:affine-coh^n}.
Explicitly, on geometric points $\CF \in \CCoh^n(\BA^d)(\bfk)$, the map does the following. The $\GL_{n}$-orbit in $[C(n,d)/\GL_{n}](\bfk)$ corresponding to $\CF$ contains a representative $(A^1,\ldots,A^d)$ such that $A^1,\ldots,A^d$ are \emph{all} upper triangular (because they pairwise commute). In particular, each $A^i$ displays its own eigenvalues on the diagonal. We have 
\[
\cycle(\CF) = \sum_{1\leq k\leq n} (A^1_{kk},\ldots,A^d_{kk}) \in \Sym^n(\BA^d).
\]
In other words, the points appearing in the support of $\CF$ can be read off from the diagonals of $A^1,\ldots,A^d$ (i.e.~they are the eigenvalues of the matrices). In particular, sheaves belonging to the punctual stack $\CCoh^n(\BA^d)_0$ correspond to tuples $(A^1,\ldots,A^d)$ as above, which are not only upper triangular, but have zeros on their diagonal. For this reason, one can present the \emph{punctual stack}
\[
\CCoh^n(\BA^d)_0=(\supp^n_{\BA^d})^{-1}(n\cdot 0),
\]
parametrising sheaves entirely supported on the origin $0 \in \BA^d$, as a global quotient stack
\[
\CCoh^n(\BA^d)_0 = [C(n,d)^{\mathsf{nilp}} / \GL_{n}],
\]
where $C(n,d)^{\mathsf{nilp}} \subset C(n,d)\subset \BA^{dn^2}$ is the closed subscheme of (pairwise commuting) nilpotent matrices.

\begin{remark}
Motzkin and Taussky proved that $C(n,2)$ is irreducible of dimension $n^2+n$ for all $n\geq 1$ \cite{MTcm}. However, $C(n,d)$ is reducible if $n,d>3$ \cite{Gerst61,Guralnick}. Baranovsky \cite{Bara1} and Basili \cite{Basi1} proved that $C(n,2)^{\mathsf{nilp}}$ is irreducible. See \cite{NS14} and the references therein for more on the state of the art on irreducibility of $C(n,d)$ and $C(n,d)^{\mathsf{nilp}}$, and for the proof of the reducibility of $C(n,d)^{\mathsf{nilp}}$ for $n,d>3$, and of the irreducibility of $C(n,3)^{\mathsf{nilp}}$ for $n\leq 6$.
\end{remark}

\section{Behaviour under \'etale maps}\label{sec:etale-behaviour}
In this section we establish a result (cf.~\Cref{lemma:etale1}), stating that the pushforward along an \'etale morphism of varieties $Y \to X$ yields an \emph{\'etale} morphism from an open substack of $\CCoh^n(Y)$ down to $\CCoh^n(X)$. This is extremely useful: it says, for instance, that if $X$ is smooth of dimension $d$, then $\CCoh^n(X)$ looks \'etale-locally like $\CCoh^n(\BA^d)$. In particular the (types of) singularities are preserved. See \Cref{thm:smoothness-of-Coh} and  \Cref{prop:local-triviality} for a couple of statements along these lines, and see \cite{Coh(X)-motivic} for a detailed proof of the isomorphism 
\[
\CCoh^n(X)_p \cong \CCoh^n(\Spec \widehat{\OO}_{X,p}),
\]
where $\CCoh^n(X)_p$ is the preimage of the $0$-cycle $n\cdot p$ along $\supp_X^n$, for a given closed point $p \in X$.

\subsection{Symmetric products and \'etale maps}
Let $f\colon Y\to X$ be an \'etale morphism of schemes; then $f$ induces (\'etale) morphisms $f^n\colon Y^n\to X^n$. We define $Y^n_{\inj}\subset Y^n$ to be the locus of points $(p_1,\ldots,p_n)\in Y^n$ such that, for $i,j\in\set{1,\ldots,n}$ we have $f(p_i)=f(p_j)$ if and only if $p_i=p_j$.

\begin{lemma}
    If $X$ is separated, then $Y^2_{\inj}$ is open in $Y^2$.
\end{lemma}
\begin{proof}
Consider the commutative diagram with cartesian square
\[
\begin{tikzcd}
Y\arrow{dr}[description]{i}\arrow[bend left=20]{drr}{f} \arrow[swap,bend right=20]{ddr}{\Delta_Y} & & \\
& Y\times_{X^2}X\MySymb{dr}\arrow{r}{g}\arrow[hook]{d} & X\arrow[hook]{d}{\Delta_X} \\
& Y\times Y\arrow[swap]{r}{f^2} & X \times X
\end{tikzcd}
\]
The complement of $Y^2_{\inj}$ in $Y\times Y$ is $Y\times_{X^2}X\setminus \Delta_Y$; since $X$ is separated, $Y\times_{X^2}X$ is closed in $Y\times Y$. The morphism $g$ is \'etale since it is a base change of $f^2$, so since $f$ is \'etale also $i$ is \'etale, hence $\Delta_Y\subset Y\times_{X^2}X$ is open and $Y\times_{X^2}X\setminus \Delta_Y$ is closed in $Y\times_{X^2}X$, which concludes the proof.
\end{proof}

\begin{prop}
If $X$ is separated, then for any $n\ge 1$, the subset $Y^n_{\inj}$ is open in $Y^n$.
\end{prop}
\begin{proof}
    For $n=1$ there is nothing to prove, since $Y^1_{\inj}=Y^1=Y$. The case $n=2$ is the previous lemma. For $n>2$, given $1\le i<j\le n$, consider the projection $\pi_{ij}\defeq(\pi_i,\pi_j)\colon Y^n\to Y^2$; the subset $Y^n_{\inj}$ is the intersection, over every possible $(i,j)$, of $\pi_{ij}^{-1}Y^2_{\inj}$, and is therefore open in $Y^n$.
\end{proof}

\begin{corollary}
The image  $\Sym^n(Y)_{\inj}$  of $Y^n_{\inj}$ in $\Sym^n(Y)$ is open.
\end{corollary}

\subsection{Finite coherent sheaves and \'etale maps}

We start with the following lemma.

\begin{lemma}\label{lemma:etale1}
Let $f\colon Y \to X$ be an \'etale map of $\bfk$-varieties. Then the direct image of coherent sheaves induces an \'etale morphism
\[
\begin{tikzcd}[row sep=tiny]
V_{\inj} \arrow{r}{f_\ast} & \CCoh^n(X)\\
{[}\CE{]} \arrow[mapsto]{r} & {[}f_\ast \CE{]}
\end{tikzcd}
\]
where $V_{\inj}\subset \CCoh^n(Y)$ is the open substack of sheaves $\CE$ on $Y$ such that $f$ is injective on $\Supp_{\mathrm{set}}(\CE)\subset Y$.
\end{lemma}

\begin{proof}
The pushforward morphism $\CCoh(Y) \to \CCoh(X)$ exists globally by \cite[\href{https://stacks.math.columbia.edu/tag/0DLX}{Tag 0DLX}]{stacks-project} (using that $f$ is quasifinite). It restricts to a $\CCoh^n(Y) \to \CCoh^n(X)$, which we now check is \'etale over $V_{\inj} \subset \CCoh^n(Y)$.

\'Etaleness of $f_\ast\colon V_{\inj} \to \CCoh^n(X)$ can be checked as follows. By \cite[\href{https://stacks.math.columbia.edu/tag/0CIK}{Tag 0CIK}]{stacks-project} it is enough to find an algebraic space $W$ along with a faithfully flat, locally finitely presented morphism $\rho\colon W \to \CCoh^n(X)$ such that the projection $W \times_{\CCoh^n(X)} V_{\inj} \to W$ is \'etale. We pick the scheme
\[
W = \Quot_X(\OO_X^{\oplus n},n),
\]
with the smooth morphism $\rho_{\OO_X^{\oplus n},n} \colon W \to \CCoh^n(X)$ from \Cref{lemma:quot-to-coh}. Note that $\rho_{\OO_X^{\oplus n},n}$ is surjective since already its restriction to $A_{X,n}$ is surjective by \Cref{prop:Coh^n-global-quotient}.  The base change
\[
W \times_{\CCoh^n(X)} V_{\inj}
\]
agrees with the open subscheme of $\Quot_{Y}(\OO_{Y}^{\oplus n},n)$ parametrising quotients $\OO_{Y}^{\oplus n} \onto \CF$ such that $f$ is injective on the set-theoretic support of $\CF$. The projection map to $W$ is the \'etale map constructed in \cite[Prop.~A.3]{BR18}. Thus we are done.
\end{proof}

\begin{corollary}\label{cor:etale2}
Let $\boldit{a} = (a_1,\ldots,a_k)$ be a sequence of positive integers summing up to $n$. Then the direct sum of coherent sheaves induces an \'etale morphism
\[
\begin{tikzcd}[row sep=tiny]
U_{\boldit{a}} \arrow{rr}{\oplus} && \CCoh^n(X) \\ ({[}\CE_1{]},\ldots,{[}\CE_k{]}) \arrow[mapsto]{rr} && {[}\CE_1\oplus \cdots \oplus \CE_k{]}
\end{tikzcd}
\]
where $U_{\boldit{a}}$ is the open substack of $\prod_{1\leq i\leq k} \CCoh^{a_i}(X)$ parametrising tuples of sheaves with pairwise disjoint support.
\end{corollary}

\begin{proof}
Consider the \'etale map $f\colon Y = \coprod_{1\leq j\leq n}X \to X$. Then we have a stratification
\[
\coprod_{\substack{k\geq 1 \\ a_1+\cdots+a_k=n}}\prod_{i=1}^k \CCoh^{a_i}(X) = \CCoh^n(Y).
\]
The open substack $V_{\inj}$ from \Cref{lemma:etale1} contains the disjoint union of the open substacks $U_{\boldit{a}}$ defined in the statement. The result follows by restriction to the $\boldit{a}$-component.
\end{proof}

The following lemma will turn out useful in the proof of \Cref{prop:local-triviality}.

\begin{lemma} \label{lemma:cart}
Let $f\colon Y\to X$ be an \'etale map of quasiprojective schemes, and denote by $\CCoh^n(Y)_{\inj}$
the open substack of $\CCoh^n(Y)$ which is the inverse image of $\Sym^n(Y)_{\inj}$.
The commutative diagram
\[
\begin{tikzcd}[column sep=large]
\CCoh^n(Y)_{\inj} \arrow{r}{f_\ast} \arrow{d}& \CCoh^n(X)\arrow{d}\\
\Sym^n(Y)_{\inj} \arrow{r}{\Sym^n(f)} & \Sym^n(X)
\end{tikzcd}
\]
is cartesian.
\end{lemma} 

\begin{proof}
Consider the induced diagram
\[
\begin{tikzcd}[row sep=large]
\CCoh^n(Y)_{\inj} \arrow{r}{\gamma}\arrow{dr}\arrow[bend left=19]{rr}[description]{f_\ast}& \CCoh^n(X)\times_{\Sym^n(X)}\Sym^n(Y)_{\inj}\MySymb{dr}\arrow{r}{\lambda}\arrow{d}& \CCoh^n(X)\arrow{d}\\
&\Sym^n(Y)_{\inj} \arrow{r}{\Sym^n(f)} & \Sym^n(X)
\end{tikzcd}
\]
and note that $\Sym^n(f)$ is \'etale. Thus its base change $\lambda$ is also \'etale. On the other hand, $f_\ast$ is \'etale by \Cref{lemma:etale1}, therefore 
\[
\begin{tikzcd}
\gamma \colon \CCoh^n(Y)_{\inj} \arrow{r} & \CCoh^n(X)\times_{\Sym^n(X)}\Sym^n(Y)_{\inj}
\end{tikzcd}
\]
is also \'etale. But $\gamma$ is bijective on $\bfk$-points, since $f$ is an isomorphism when restricted to fully punctual $0$-dimensional closed subschemes. In a little more detail, let 
\[
\left(\CG,\sum_ja_jy_j\right)\in \CCoh^n(X)\times_{\Sym^n(X)}\Sym^n(Y)_{\inj}
\]
be a $\bfk$-point of the fibre product. By \Cref{basic-lemma} we have canonically defined closed subschemes $\iota_j \colon W_j \into X$ such that $\CG = \bigoplus_j \iota_{j\ast}\CG|_{W_j}$ with $W_{j,\red} = x_j$ for some point $x_j \in X$, $f(y_j)=x_j$ and $\CG|_{W_j}$ of length $a_j$ for all $j$. For any $j$ there is a unique closed subscheme $Z_j \into Y$ supported in $y_j$, and such that $f(Z_j) = W_j$. The induced morphism $f_j\colon Z_j \to W_j$, namely the restriction of $f$, is an isomorphism. It is then clear that the only $\bfk$-point in $\CCoh^n(Y)_{\inj}$ mapping to $(\CG,\sum_ja_jy_j)$ under $\gamma$ is the one corresponding to $\CF = \bigoplus_j \CF_j$, where $\CF_j = f_j^\ast \CG|_{W_j}$.

Finally, $\gamma$ is an isomorphism on stabiliser groups, since, in the above notation, we have
\[
\Hom(\CF,\CF)=\bigoplus_j\Hom(\CF_j,\CF_j)=\bigoplus_j \Hom(\CG|_{W_j},\CG|_{W_j})=\Hom(\CG,\CG).
\]
We conclude that $\gamma$ is \'etale, bijective on points and on automorphism groups, hence an isomorphism by  \cite[\href{https://stacks.math.columbia.edu/tag/0DUD}{Tag 0DUD}]{stacks-project}, as required.
\end{proof}

\subsection{Smoothness of \texorpdfstring{$\CCoh^n(X)$}{}}\label{sec:smoothness-of-Coh}
Let $X$ be a smooth $d$-dimensional $\bfk$-variety. By the \'etale-local structure of $\CCoh^n(X)$, to check whether the stack is smooth it is enough to check whether $\CCoh^n(\BA^d)$ is smooth. In a little more detail, one needs to pick \'etale coordinates on $X$ and apply \Cref{lemma:etale1}. One then obtains the following result.

\begin{theorem}\label{thm:smoothness-of-Coh}
Let $X$ be a smooth $d$-dimensional $\bfk$-variety. Then $\CCoh^n(X)$ is smooth if and only if either $d=1$ or $n=1$.
\end{theorem}

\begin{proof}
The stack $\CCoh^n(\BA^1)$ is smooth for all $n$ by \Cref{thm:affine-coh^n}. The stack $\CCoh^1(\BA^d)$ is smooth for all $d$ by \Cref{ex:few-points}. On the other hand, if $n,d>1$, then $\CCoh^n(\BA^d)$ is singular, because the commuting variety $C(n,d)$ is singular \cite{Hreinsdottir}.
\end{proof}

\subsection{Zariski triviality over strata}\label{sec:stratification}

Let $X$ be a variety. There is a locally closed stratification of the symmetric product 
\begin{equation}\label{eqn:sm-strata-alpha}
\Sym^n(X) = \coprod_{\alpha\vdash n}\Sym^n_\alpha(X)
\end{equation}
by partitions $\alpha$ of the integer $n$. We use the notation $\alpha = (1^{\alpha_1}\cdots i^{\alpha_i}\cdots n^{\alpha_n})$ for a partition, meaning that $\alpha$ consists of $\alpha_i$ parts of size $i$ for every $i=1,\ldots,n$. Of course $\alpha_i$ might be $0$ for one or more values of $i$ (in which case we omit writing `$i^{\alpha_i}$'). Each partition specifies the multiplicities of the supporting points of a given $0$-cycle. For instance, $\Sym_{(1^22^13^1)}^7(X) \into \Sym^7(X)$ parametrises $0$-cycles of the form $p_1+p_2+2p_3+3p_4$, where $p_i\neq p_j$ for every $i\neq j$.

The stratification \eqref{eqn:sm-strata-alpha} allows us to define, for each partition $\alpha\vdash n$, the locally closed substack 
\[
\begin{tikzcd}
\CCoh^n_\alpha(X) = \CCoh^n(X) \times_{\Sym^n(X)}\Sym^n_\alpha(X) \arrow[hook]{r} & \CCoh^n(X),
\end{tikzcd}
\]
parametrising sheaves whose support is distributed according to $\alpha$.
Note that, for any given point $p \in X$, the closed substack
\[
\CCoh^n(X)_p = (\supp_X^n)^{-1}(n\cdot p) \subset \CCoh^n(X)
\]
is naturally a closed substack of the stratum $\CCoh^n_{(n^1)}(X) = \CCoh^n(X) \times_{\Sym^n(X)}X$, where we view $X \into \Sym^n(X)$ embedded as the diagonal, sending $p \mapsto n\cdot p$.

\begin{theorem}\label{prop:local-triviality}
Let $X$ be a smooth variety of dimension $d$.
\begin{itemize}
\item [\mylabel{punct-1}{(1)}] If $X=\BA^d$, the projection
\[
\begin{tikzcd}
\supp^n_{\BA^d} \colon \CCoh^n_{(n^1)}(\BA^d) \arrow{r} & \BA^d
\end{tikzcd}
\]
is a \emph{trivial} fibration with fibre $\CCoh^n(\BA^d)_0$, where $0 \in \BA^d$ denotes the origin.
\item [\mylabel{punct-2}{(2)}] The projection 
\[
\begin{tikzcd}
\supp^n_{X} \colon \CCoh^n_{(n^1)}(X) \arrow{r} & X
\end{tikzcd}
\]
is \emph{Zariski locally trivial} with fibre $\CCoh^n(\BA^d)_0$.
\item [\mylabel{punct-3}{(3)}] If $\alpha = (1^{\alpha_1}\cdots i^{\alpha_i}\cdots n^{\alpha_n})$ is a partition of $n$, then
\[
\begin{tikzcd}
\supp^n_X \colon \CCoh^n_\alpha(X) \arrow{r} & \Sym^n_\alpha (X)
\end{tikzcd}
\]
is \emph{\'etale locally trivial} with fibre $\prod_i \CCoh^i(\BA^d)_0^{\alpha_i}$.
\end{itemize}
\end{theorem}

\begin{proof}
We proceed step by step.

\smallbreak
\noindent
\emph{Proof of }\ref{punct-1}. The family of translations parametrised by $\BA^d$ induces a morphism of stacks
\[
\begin{tikzcd}
    t\colon\BA^d \times \CCoh^n(\BA^d)_0 \arrow{r} & \CCoh^n_{(n^1)}(\BA^d).
\end{tikzcd}
\]
We now check $t$ is an isomorphism. Consider an $S$-valued point of the source, namely a pair $(f,\cF)$ where $f \colon S \to \BA^d$ is a morphism from a $\bfk$-scheme and $\cF \in \Coh(\BA^d \times S)$ is a family of $0$-dimensional sheaves on $\BA^d$, supported entirely at the origin $0 \in \BA^d$. Consider the morphism
\[
\begin{tikzcd}
f_0 \colon \BA^d \times S \arrow{r}{\sim} & \BA^d \times S, \quad (x,s) \mapsto (x+f(s),s).
\end{tikzcd}
\]
The pushforward $f_{0,\,\ast}\cF$ is still $S$-flat and defines an $S$-valued point of $\CCoh^n_{(n^1)}(\BA^d)$. Conversely, given $\cG \in \CCoh^n_{(n^1)}(\BA^d)(S)$, consider the morphism $f_\cG \colon S \to \BA^d$ sending $s \mapsto p_s$, where $\Supp_{\mathrm{set}}(\cG_s) = \set{p_s}$. Now pullback $\cG$ along the isomorphism
\[
\begin{tikzcd}
\BA^d \times S \arrow{r}{\sim} & \BA^d \times S, \quad (x,s) \mapsto (x-p_s,s).
\end{tikzcd}
\]
This gives an $S$-valued point of $\CCoh^n(\BA^d)_0$, which considered together with $f_{\cG}$ yields an $S$-valued point of $\BA^d \times \CCoh^n(\BA^d)_0$. This correspondence establishes an equivalence of groupoids for every $\bfk$-scheme $S$, thus $t$ is an isomorphism of stacks. It is also clear that $\supp_X^n \circ t$ agrees with the first projection onto $\BA^d$. 

\smallbreak
\noindent
\emph{Proof of }\ref{punct-2}. We take \'etale coordinates on $X$, i.e.~we cover $X$ with Zariski open subsets $U\subset X$ equipped with \'etale maps
\[
\begin{tikzcd}
U \arrow{r}{\varepsilon} & \BA^d.  
\end{tikzcd}
\]
Consider the open substack $V_{\inj} \subset \CCoh^n(U)  \subset \CCoh^n(X)$ of sheaves $[\CF] \in \CCoh^n(U)$ such that $\varepsilon|_{\Supp_{\mathrm{set}}(\CF)}$ is injective. By \Cref{lemma:etale1}, the pushforward along $\varepsilon$ yields an \'etale map
\[
\begin{tikzcd}
V_{\inj} \arrow{r}{\varepsilon_\ast} & \CCoh^n(\BA^d).
\end{tikzcd}
\]
The punctual stack $\CCoh^n_{(n)}(U)$ sits inside $V_{\inj}$ as a closed substack, and by construction $\varepsilon_\ast$ restricts to a map
\[
\begin{tikzcd}
\CCoh^n_{(n)}(U)\arrow{r}{\varepsilon_\ast} & \CCoh^n_{(n)}(\BA^d).
\end{tikzcd}
\]
We claim that the commutative square
\[
\begin{tikzcd}
\CCoh^n_{(n)}(U)\arrow{r}{\varepsilon_\ast}\arrow[swap]{d}{\supp^n_U} & \CCoh^n_{(n)}(\BA^d)\arrow{d}{\supp^n_{\BA^d}} \\
U \arrow{r}{\varepsilon} & \BA^d
\end{tikzcd}
\]
is cartesian. Granting this, by \ref{punct-1} we know that the map $\supp^n_{\BA^d}$ is Zariski trivial with fibre $\CCoh^n(\BA^d)_0$, therefore by pullback we obtain the claim. It thus remains to confirm that there is an isomorphism
\[
\begin{tikzcd}
\CCoh^n_{(n)}(U) \arrow{r}{\sim} & \CCoh^n_{(n)}(\BA^d)\times_{\BA^d}U.
\end{tikzcd}
\]
Note that there are natural cartesian diagrams
\begin{equation}
\begin{tikzcd}
U\MySymb{dr} \arrow{r}{\varepsilon} \arrow{d}& \BA^d\arrow{d}\\
\Sym^n(U)_{\inj} \arrow{r} & \Sym^n(\BA^d)
\end{tikzcd}
\qquad 
\begin{tikzcd}
\CCoh^n(U)_{\inj}\MySymb{dr} \arrow{r}{\varepsilon_\ast} \arrow{d}& \CCoh^n(\BA^d)\arrow{d}\\
\Sym^n(U)_{\inj} \arrow{r} & \Sym^n(\BA^d)
\end{tikzcd}
\end{equation}
where the rightmost one was obtained in \Cref{lemma:cart}. We compute
\begin{align*}
\CCoh^n_{(n)}(\BA^d)\times_{\BA^d}U
&=(\CCoh^n(\BA^d) \times_{\Sym^n(\BA^d)}\BA^d) \times_{\BA^d}(\BA^d \times_{\Sym^n(\BA^d)} \Sym^n(U)_{\inj}) \\
&=\CCoh^n(\BA^d) \times_{\Sym^n(\BA^d)} \BA^d \times_{\Sym^n(\BA^d)}\Sym^n(U)_{\inj}\\
&=\CCoh^n(\BA^d) \times_{\Sym^n(\BA^d)}\Sym^n(U)_{\inj}\times_{\Sym^n(\BA^d)}\BA^d \\
&=\CCoh^n(U)_{\inj}\times_{\Sym^n(\BA^d)}\BA^d \\
&=\CCoh^n_{(n)}(U),
\end{align*}
as required.

\smallbreak
\noindent
\emph{Proof of }\ref{punct-3}.
Fix a partition $\alpha$ of $n$. 
Let $V_\alpha \subset \prod_i \CCoh^i(X)^{\alpha_i}$ be the open substack of tuples of sheaves with pairwise disjoint support. Then, by \Cref{cor:etale2}, the direct sum map $V_\alpha \to \CCoh^n(X)$ is an \'etale morphism, and we let $U_\alpha$ denote its image. Consider the base change diagram
\[
\begin{tikzcd}
P_\alpha\MySymb{dr} \arrow{d}\arrow[hook]{r} & V_\alpha \arrow{d}{\textrm{\'etale}} \\
\CCoh^n_\alpha(X)\arrow[hook]{r} & U_\alpha 
\end{tikzcd}
\]
defining the stack $P_\alpha$. Let $W_\alpha \subset \prod_i X^{\alpha_i}$ be the open subscheme parametrising tuples of pairwise distinct points, namely
\[
W_\alpha = \Set{(p_k)_k \in X^{\sum_i\alpha_i}|p_k\neq p_h\mbox{ if }k\neq h}.
\]
Then $P_\alpha$ maps to $W_\alpha$. Indeed, it agrees with the open substack of $\prod_i \CCoh^i_{(i^1)}(X)^{\alpha_i}$ parametrising tuples of (punctual) sheaves with pairwise disjoint support. In particular, the diagram
\[
\begin{tikzcd}
    P_\alpha\arrow[hook]{r}\arrow{d} & \prod_i \CCoh^i_{(i^1)}(X)^{\alpha_i}\arrow{d}{\pi_\alpha} \\
    W_\alpha\arrow[hook]{r} & \prod_i X^{\alpha_i}
\end{tikzcd}
\]
is cartesian.
The automorphism group $G_\alpha$ of the partition $\alpha$ acts freely on $W_\alpha$ and thus yields a surjective \'etale map $W_\alpha \to \Sym^n_\alpha(X)$. Note, also, that $P_\alpha$ is also the fibre product
\[
\begin{tikzcd}
\CCoh^n_\alpha(X)\MySymb{dr} \arrow[swap]{d}{\supp_X^n} & P_\alpha\MySymb{dr} \arrow{l} \arrow{d} \arrow[hook]{r} & \prod_i \CCoh^i_{(i^1)}(X)^{\alpha_i} \arrow{d}{\pi_\alpha} \\
\Sym^n_\alpha(X) & W_\alpha\arrow{l}{\textrm{\'et}} \arrow[hook]{r} & \prod_i X^{\alpha_i}
\end{tikzcd}
\]
where $\pi_\alpha$ is Zariski locally trivial with fibre $\prod_i \CCoh^i(\BA^d)_0^{\alpha_i}$ by \ref{punct-2}, which proves that so is the map $P_\alpha \to W_\alpha$. Reading the leftmost cartesian square completes the proof of the statement.
\end{proof}


\bibliographystyle{amsplain-nodash} 
\bibliography{bib}

\bigskip
\noindent
{\small Andrea T. Ricolfi \\
\address{SISSA, Via Bonomea 265, 34136, Trieste (Italy)} \\
\href{mailto:aricolfi@sissa.it}{\texttt{aricolfi@sissa.it}}
}

\bigskip
\noindent
{\small Barbara Fantechi \\
\address{SISSA, Via Bonomea 265, 34136, Trieste (Italy)} \\
\href{mailto:aricolfi@sissa.it}{\texttt{fantechi@sissa.it}}
}

\end{document}